\documentclass[a4paper,12pt]{article}
\usepackage{amsmath}
\usepackage{amsthm}
\usepackage{epsfig}
\usepackage{psfrag}
\usepackage[mathscr,mathcal]{eucal}
\usepackage{amssymb}
\usepackage{enumerate}
\def \N {\mathbb N}

\def \R {\mathbb R}

\def \D {\mathbb D}

\def \Ordo {{\cal O}}
\def\ordo{o}

\def \ind{1\!\!1}

\def\phi{\varphi}

\def\be{\begin{equation}}
\def\ee{\end{equation}}

\def\bea{\begin{eqnarray}}
\def\eea{\end{eqnarray}}

\newcommand{\abs}[1]{\left|{#1}\right|}

\newcommand{\prob}[1]{\ensuremath{\mathbf{P}\left(#1\right)}}

\newcommand{\expect}[1]{\ensuremath{\mathbf{E}\left(#1\right)}}

\newcommand{\condprob}[2]{\ensuremath{\mathbf{P}\left(#1\,\big|\,#2\right)}}
\newcommand{\condexpect}[2]{\ensuremath{\mathbf{E}\left(#1\,\big|\,#2\right)}}
\newcommand{\condvar}[2]{\ensuremath{\mathbf{Var}\left(#1\,\big|\,#2\right)}}
\newcommand{\condcov}[3]{\ensuremath{\mathbf{Cov}\left(#1,#2\,\big|\,#3\right)}}

\def\cD{\mathcal{D}}
\def\cE{\mathcal{E}}
\def\cF{\mathcal{F}}
\def\cG{\mathcal{G}}

\def\cV{\mathcal{V}}
\def\cW{\mathcal{W}}

\def \rd{\cD} 
\def \mypoi{\mathbf{p}}
\def \mybin{\mathbf{b}}
\def \mygamma{\mathbf{g}}
\def \mysor{\mathbf{q}}

\def \toind {\buildrel {\text{d}}\over{\longrightarrow}}
\def \toinp {\buildrel {\text{p}}\over{\longrightarrow}}
\def \mybart { t^*}

\newcommand{\semmi}[1]{}

\def \M       {\mathcal{M}}
\def \A       {\mathcal{A}}

\newtheorem {theorem}{Theorem}
\newtheorem {lemma}{Lemma}[section]

\newtheorem {definition}{Definition}[section]

\title{Time evolution of dense multigraph limits under
 edge-conservative preferential attachment dynamics}

\author{ Bal\'azs R\'ath\thanks{ETH Z\"urich, Department of Mathematics, R\"amistrasse 101, 8092 Z\"urich. 
Email: rathb@math.ethz.ch.}}

\begin{document}

\maketitle


\medskip

\begin{abstract}
We define the edge reconnecting model, a random multigraph evolving in time. 
At each time step we change one endpoint of a uniformly chosen edge: the 
new endpoint is chosen by linear preferential attachment.
We consider a sequence of edge reconnecting models where the sequence of 
initial multigraphs is convergent in a sense which is a 
natural generalization of the notion of convergence of dense
 graph sequences, defined by Lov\'asz and Szegedy in \cite{Lovasz_Szegedy_2006}.
 We investigate how the limit object evolves under the edge 
reconnecting dynamics if we rescale time properly:
 we give the complete characterization of the time evolution 
of the limit object from its initial state up to the
 stationary state, which is described in the companion paper \cite{RB_SZL}.
In our proofs we use the theory of exchangeable arrays, 
queuing and diffusion processes.
 The number of parallel edges and the degrees evolve on different 
timescales and because of this the model exhibits subaging.
\end{abstract}

$ $

\section{Introduction}

We introduce the \emph{edge reconnecting model}, a random
 multigraph (undirected graph with multiple and loop edges) evolving in time.
 Denote the multigraph at time $T$ by $\cG_n(T)$, where $T=0,1,2,\dots$ and $n=\abs{V(\cG_n(T))}$ is the number of vertices.
 We denote by $m=\abs{E(\cG_n(T))}$ the number of edges (the number of vertices and edges does not change over time).
Given the multigraph $\cG_n(T)$ we get $\cG_n(T+1)$ by uniformly choosing an edge in $E(\cG_n(T))$, 
choosing one of the endpoints of that edge with a coin flip and reconnecting the edge to a new endpoint
 which is chosen using the rule of linear preferential attachment: a vertex $v$ is chosen with probability
 $\frac{d(v)+\kappa}{2m+n\kappa}$, where $d(v)$ is the degree of vertex $v$ in $\cG_n(T)$ and $\kappa \in (0,+\infty)$
  is a fixed parameter of the model.

Our aim is to describe the time evolution of the edge reconnecting model $\cG_n(T)$ when $1 \ll n$ using
 the terminology of \emph{dense graph limits}. The notion of convergence of simple graph sequences was defined and
several equivalent characterizations of \emph{graphons} (limit objects of convergent simple graph sequences) were given
in \cite{Lovasz_Szegedy_2006}.
In \cite{KI_RB}  we give a natural generalization of the theory of dense graph limits to multigraphs
(see also \cite{Lovasz_Szegedy_2010} for similar results in a more general setting), which we briefly recall now.

 Denote by $\M$ the set of multigraphs.
If $G \in \M$ and $v,w \in V(G)$, denote by
 $E(v,w)$ the number of edges between $v$ and $w$ in $F$ (loop edges count twice).
For $F,G \in \M$ we define the density of copies of $F$ in $G$ by the formula
\begin{equation*}
 t_{=}(F,G)= \frac{1}{{\abs{V(G)} }^{\abs{V(F)}}}  \sum_{\varphi: V(F) \to V(G)}
\ind[ \forall \; v,w \in V(F) : \; E(v,w)=E(\varphi(v),\varphi(w)) ].
\end{equation*}
We say that a sequence of 
multigraphs $\left(G_n\right)_{n=1}^{\infty}$ is convergent
if for every $F \in \M$ the limit $g(F)=\lim_{n \to \infty}t_{=}(F,G_n)$ exists
 and 
 $g(\cdot)$ is a ``non-defective probability distribution'' on the set of multigraphs
 (see  Subsection  \ref{subsection_multi_def} for details).
In plain words: the sequence 
$\big(G_n\big)_{n=1}^{\infty}$ is convergent if the density of every fixed graph $F$ in $G_n$ 
converges as $n \to \infty$, and
``no mass escapes to infinity'' during this limiting procedure.

  The definition of the limit objects of convergent multigraph sequences is slightly more complicated than that of graphons.
A  measurable function
 $W:[0,1] \times [0,1] \times \N_0 \to [0,1]$  satisfying
\begin{equation}\label{def_eq_W_intro}
W(x,y,l) \equiv W(y,x,l), \quad
 \sum_{l=0}^{\infty} W(x,y,l) \equiv 1, \quad
 W(x,x, 2l+1) \equiv 0
\end{equation}
is called a \emph{multigraphon}.
Note that $\left( W(x,y,l) \right)_{l=0}^{\infty}$ is a probability distribution on $\N_0$ for each $x,y \in [0,1]$.
We say that $G_n \to W$ if for every  $F \in \M$ with $V(F)=\{1,\dots,k\}$ 
we have $\lim_{n \to \infty}t_{=}(F,G_n)=t_{=}(F,W)$ where
\begin{equation*}
 t_{=}(F,W):=\int_{[0,1]^k} \prod_{v\leq w \leq k}
W(x_v,x_w,E(v,w))\, \mathrm{d} x_1\, \mathrm{d} x_2\, \dots\, \mathrm{d}x_k.
\end{equation*}
 \cite[Theorem 1]{KI_RB} states that if a sequence of multigraphs $\left(G_n\right)_{n=1}^{\infty}$ 
is convergent then $G_n \to W$ for
some multigraphon $W$ and conversely, every multigraphon $W$ arises this way. 
We say that a sequence of random multigraphs  $\big(\cG_n\big)_{n=1}^{\infty}$ converges in probability  to
a  multigraphon $W$ (or briefly write $\cG_n \toinp W$) if for every simple graph $F$ we have
 $t_{=}(F,\cG_n) \toinp t_{=}(F,W)$, i.e.
\begin{equation}\label{convergence_of_random_multigraphons}
 \forall \, F \in \M \; \forall \, \varepsilon>0:  \;
\lim_{n \to \infty} \prob{ \abs{ t_{=}(F, \cG_n) - t_{=}(F, W) }>\varepsilon}=0.
\end{equation}

\medskip

 In  \cite[Lemma 2.1]{RB_SZL} we build on methods and models of  \cite[Section 3.4]{randomlygrown} to
 explicitly describe the unique stationary distribution $\cG_n(\infty)$ of the edge
 reconnecting model and in  \cite[Theorem 2]{RB_SZL} we prove that there is a multigraphon $\hat{W}_{\infty}$ such that
\begin{equation}\label{stac_multigraphon_intro}
\cG_n(\infty) \toinp \hat{W}_{\infty}, \qquad     n \to \infty
\end{equation}
under the condition that
  $m \approx \frac12 \rho n^2$, where $\rho \in (0,+\infty)$ 
 is a fixed parameter of the model called the \emph{edge density}. The form of the limiting multigraphon $\hat{W}_{\infty}$
depends on  $\rho$ and the linear preferential attachment parameter  $\kappa$.

\medskip

Now we describe the main results of this paper:
if we consider a sequence of edge reconnecting models with a convergent sequence of initial multigraphs
$\cG_n(0) \to W$ (satisfying some extra regularity conditions), then for every $t \in (0,+\infty)$ we have
\begin{equation}\label{intro_conv_theorems}
\cG_n(t \cdot n^2) \toinp \breve{W}_t \qquad \text{ and } \qquad
\cG_n(t \cdot n^3) \toinp \tilde{W}_t, \qquad n\ \to \infty
\end{equation}
 where the multigraphons $\breve{W}_t$ and $\tilde{W}_t$ are explicit, continuous functions of  $t$, the initial multigraphon $W$
 and  $\kappa$. Moreover we have
\begin{equation}\label{full_characterization}
 \lim_{t \to 0_+} \breve{W}_t=W, \qquad
\lim_{t \to \infty} \breve{W}_t =\lim_{t \to 0_+} \tilde{W}_t, \qquad
\lim_{t \to \infty} \tilde{W}_t =\hat{W}_{\infty}
\end{equation}
where $\hat{W}_{\infty}$ is the multigraphon in \eqref{stac_multigraphon_intro}.
 Thus by \eqref{full_characterization} the convergence theorems  \eqref{intro_conv_theorems}
give the full characterization of the time evolution of the multigraphons arising as the graph limits of the
edge reconnecting model.

$ $

Although our theorems are stated using the ``multigraphon" formalism,
 in their proofs we use the correspondence between the theory of graph limits and that of exchangeable arrays,
 a connection first observed in \cite{diaconis_janson}.
  The basic idea of the proof of our main theorems is to relate the time evolution of the edge 
reconnecting model to certain continuous-time stochastic processes using an appropriate rescaling of time:
\begin{itemize}
\item If we fix a vertex $v \in V(\cG_n(0))$ and denote by $d(T,v)$ the degree of $v$ in $\cG_n(T)$ then the evolution of the
$\R_+$-valued continuous-time stochastic process $\frac{1}{n} d( n^3 \cdot t,v)$ ``almost looks like" that of a
 Cox-Ingersoll-Ross process (a diffusion process that is commonly used in financial mathemathics to model the evolution of interest rates).
 This fact is rigorously proved using the theory of stochastic differential equations and is used in the proof 
of $\cG_n(t \cdot n^3) \toinp \tilde{W}_t$.

\item If we fix two vertices $v,w \in V(\cG_n(0))$ and denote by $E(T,v,w)$ the number of parallel/loop edges
 connecting $v$ and $w$ in $\cG_n(T)$ then the evolution of 
the $\N_0$-valued continuous-time stochastic process $E(n^2\cdot t, v,w)$ 
``almost looks like" that of the queue length of an M/M/$\infty$-queue.
 This fact is rigorously proved using a coupling argument and is
  used in the proof of
$\cG_n(t \cdot n^2) \toinp \breve{W}_t$.
\end{itemize}

The most interesting property of the edge reconnecting model is the separation of \emph{two different timescales} in
\eqref{intro_conv_theorems} and \eqref{full_characterization}: the degrees of the vertices only change significantly on 
the $n^3$ timescale, whereas the number of parallel (or loop) edges between two vertices evolves on the much
 faster $n^2$ timescale. 
The arrival rate of the M/M/$\infty$-queue describing the evolution of $E(n^2\cdot t, v,w)$ 
 depends on the current degrees of $v$ and $w$ 
(if their degrees are high then edges appear between them with higher probability
 due to preferential attachment), but since the degrees evolve on the much slower $n^3$ timescale,
 they may be treated as  constant background parameters on the $n^2$ timescale. 
The stochastic process
 $E( n^3 \cdot t+ n^2 \cdot s,v,w)$ looks stationary in the time variable $s\in \R$ if $t \in (0, +\infty)$ is fixed and
 $1 \ll n$, but different values of $t$ yield distinct pseudo-stationary
distributions since  $n^3\cdot (t_2-t_1)$ steps are enough for the background variables (degrees) to significantly change.
This phenomenon is called \emph{subaging} in  \cite{jeanbertoin}.

A similar dynamical random graph model where no subaging occurs
 is studied in the context of equation-free numerical methods 
 in \cite{katy} and \cite{yannis_karthik_katy_balazs}.

\medskip
The rest of this paper is organized as follows:

In Section \ref{section_notations_statements} we introduce some notation,
precisely formulate the above stated results, and give some heuristic hints on their proofs.

In Section \ref{section_vert_ex} we relate the theory of multigraph limits to exchangeable  arrays.

In Section \ref{section_technical_bounds} we prove some technical lemmas showing that degrees and multiple edges in the
edge reconnecting model are well-behaved.

In Section \ref{section_proof_of_thm_elek} we prove the rigorous version of 
$\cG_n(t \cdot n^2) \toinp \breve{W}_t$, Theorem \ref{theorem_elek_fejlodese}.

In Section \ref{section_fokok_fejl} 
we prove the rigorous version of $\cG_n(t \cdot n^3) \toinp \tilde{W}_t$, Theorem \ref{theorem_fokszamok_fejlodese}.

$ $

\noindent
{\bf Acknowledgement.}
The author thanks L\'aszl\'o Lov\'asz and Ioannis Kevrekidis for posing the research problem that 
became the subject of this paper.
The comments of the anonymous referees helped a lot in the developement of the paper to its current form.

The research of the author was
partially supported by the OTKA (Hungarian
National Research Fund) grants
K 60708 and CNK 77778, Morgan Stanley Analytics Budapest and
 the grant ERC-2009-AdG 245728-RWPERCRI.

\section{Notations, definitions, theorems}\label{section_notations_statements}

This section is organized as follows:

In Subsection \ref{subsection_edge_reconnecting_model} we precisely define the edge reconnecting model.

In Subsection \ref{subsection_multi_def} we give a probabilistic meaning to $t_=(F,W)$ by introducing $W$-random multigraphs and also define
the average degree $D(W,x)$ of $W$ at point $x$.

In Subsection \ref{subsection_aux} we recall some relevant properties of the M/M/$\infty$-queue and the Cox-Ingersoll-Ross process.

In Subsection \ref{subsection_statements_of_thms} we state Theorem \ref{theorem_elek_fejlodese} and
Theorem \ref{theorem_fokszamok_fejlodese}, and we also give some heuristic comments on ideas behind their proofs.  

In Subsection \ref{subsection_properties_W} we derive \eqref{full_characterization} and relate some properties of the  multigraphons 
from our main theorems to the \emph{configuration model}.

\medskip
Denote by $\N_0=\{0,1,2,\dots\}$ and $[n]:=\{1,\dots,n \}$.
Denote by $\M$ the set of undirected multigraphs (graphs with multiple and loop edges) and by $\M_n$ the set of multigraphs on $n$ vertices.
Let $G \in \M_n$. The adjacency matrix of a labeling of the multigraph $G$ with $[n]$
is denoted by $\left(B(i,j)\right)_{i,j=1}^n$, where $B(i,j) \in \N_0$ is the
number of edges connecting the vertices labeled by $i$ and $j$.
$B(i,j)=B(j,i)$ since the graph is undirected
 and  $B(i,i)$ is  two times the number of loop edges at vertex $i$ (thus $B(i,i)$ is an even number).
 An unlabeled multigraph is the equivalence class of labeled multigraphs where two labeled graphs are equivalent if one can be
 obtained by relabeling the other.
  Thus $\M$ is the set of these equivalence classes of labeled multigraphs, which are also called isomorphism types.
 We denote the set of adjacency matrices of multigraphs on $n$ nodes by $\A_n$, thus
 \[\A_n = \left\{  B\in \N_0^{n \times n}\, :\, B^T=B, \, \forall \, i \in [n] \,\;\;
   2\, |\,   B(i,i)  \right\}. \]

The degree of the vertex labeled by $i$ in $G$ with adjacency matrix $B \in \A_n$ is defined by
$\label{def_degree_B}
d(B,i):= \sum_{j=1}^n B(i,j)$, thus $d(B,i)$ is the number of edge-endpoints at $i$ (loop edges count twice).
 Let $m= \frac{1}{2} \sum_{i, j=1}^n B(i,j)=\frac{1}{2} \sum_{i=1}^n d(B,i)$
 denote the number of edges. Denote by $\A_n^m$ the set of adjacency matrices on $n$ vertices with $m$ edges.

We denote a random element of $\A_n$ by $\mathbf{X}_n$. We may associate a random multigraph $\cG_n$  to $\mathbf{X}_n$
by taking the isomorphism class of $\mathbf{X}_n$.

 We use the standard notation $\mathbf{X}_n\sim \mathbf{X}'_n$ if
 $\mathbf{X}_n$ and $\mathbf{X}'_n$ are 
identically distributed, i.e.
\[\forall \, A \in \A_n: \; \prob{ \mathbf{X}_n=A}= \prob{ \mathbf{X}'_n=A}.\]

If $\mathbf{X}_n$ is a random element of $\A_n$ then
\[\mathbf{X}_n^{[k]}:=\left( X_n(i,j) \right)_{i,j=1}^k\] is a random element of $\A_k$.

\subsection{The edge reconnecting model}\label{subsection_edge_reconnecting_model}

Now we describe the dynamics of the edge reconnecting model, which is a discrete time Markov chain with state space $\A_n^m$:
 neither the number of vertices, nor the number of edges is changed by the dynamics. 
 $\mathbf{X}(T)=\left(X(T,i,j)\right)_{i,j=1}^n$ is the state of our Markov chain at time $T$.

Given the adjacency matrix $\mathbf{X}(T)$ we get $\mathbf{X}(T+1)$ in the following way:
let $\kappa \in (0, +\infty)$. We choose a random vertex $\cV_{old}(T)$
with distribution
\begin{equation}\label{Vold}
 \condprob{\cV_{old}(T) =
i}{\mathbf{X}(T)}=\frac{d(\mathbf{X}(T),i)}{2m}
 \end{equation}

  Then we choose a uniform edge $\cE_{old}(T)=\{\cV_{old}(T),\cW(T) \}$
  going out of $\cV_{old}(T)$:
  \begin{equation*}
  \condprob{\cW(T)=i}{\mathbf{X}(T),\cV_{old}(T)}= \frac{X(T,\cV_{old}(T),i) }{d(\mathbf{X}(T),\cV_{old})}
   \end{equation*}
Note that $\cE_{old}(T)$ is uniformly distributed over all edges
of the graph at time $T$ and given $\cE_{old}(T)$, $\cV_{old}(T)$ is
uniformly chosen from the endvertices of $\cE_{old}(T)$. Moreover
 \begin{equation}\label{cW_Vold}
 \condprob{\cW(T) =
i}{\mathbf{X}(T)}=\frac{d(\mathbf{X}(T),i)}{2m}.
 \end{equation}

 Given
$\mathbf{X}(T)$, choose $\cV_{new}(T)$ according to the rules of linear
preferential attachment:
\begin{equation}\label{Vnew}
\condprob{\cV_{new}(T)=i}{\mathbf{X}(T),\cV_{old}(T),\cW(T)}=
\frac{d(\mathbf{X}(T),i)+\kappa}{2m+n\kappa}.
 \end{equation}
Thus $\cV_{new}(T)$ is conditionally independent from $\cV_{old}(T)$ and
$\cW(T)$ given $\mathbf{X}(T)$.

 Let $\cE_{new}(T):=\{\cV_{new}(T),\cW(T)
\}$.

One step of the Markov chain consists of replacing the edge
$\cE_{old}(T)$ with $\cE_{new}(T)$:
\begin{multline}\label{edge_evolution_ind}
X(T+1,i,j)=X(T,i,j)-\ind[\cV_{old}(T)=i, \cW(T)=j]-\ind[\cV_{old}(T)=j, \cW(T)=i]+\\
\ind[\cV_{new}(T)=i, \cW(T)=j]+\ind[\cV_{new}(T)=j, \cW(T)=i]
\end{multline}

This Markov chain is easily seen to be irreducible and aperiodic on $\A_n^m$. 
Note that  for any $k\leq n$
the $\N_0^{[k]}$-valued stochastic process $\left(d(\mathbf{X}(T),i)\right)_{i=1}^k$, $T=0,1,\dots$ is itself a Markov
 chain.

\subsection{Multigraphons and $W$-random multigraphs}\label{subsection_multi_def}

In this subsection we give a probabilistic meaning to $t_=(F,W)$ by introducing $W$-random multigraphs and also define
the average degree $D(W,x)$ of $W$ at point $x$.
Note that the notion of the $W$-random graph (see Definition \ref{def_X_W})
 is already present in \cite{Lovasz_Szegedy_2006}.

 Suppose $F \in \M_k,$  $G \in \M_n$ and denote by $A \in \A_k$ and $B \in \A_n$ the adjacency matrices of $F$ and $G$.
  If  $g: \M \to \R$ then we say that $g$ is a multigraph parameter. Let
$g(A):=g(F)$. 
Conversely, if $g: \bigcup_{k=1}^{\infty} \A_k \to \R$ is constant on isomorphism classes, then $g$ defines a multigraph parameter.

We  define the \emph{induced homomorphism density} of $F$ into $G$ by
  \begin{equation*}
    t_{=}(F,G):=t_{=}(A,B):= \frac{1}{n^k} \sum_{\varphi:[k] \rightarrow [n]}
     \ind \left[ \, \forall i,j \in [k]:\; A(i,j)= B(\varphi(i),\varphi(j))\right] .
  \end{equation*}

We say that a sequence of multigraphs  $\left(G_n\right)_{n=1}^{\infty}$ is convergent
if for every $k \in \N$ and every multigraph $F \in \M_k$ the limit $g(F)=\lim_{n \to \infty}t_{=}(F,G_n)$ 
exists, and
 we have $\sum_{A \in \A_k} g(A)=1$. 
 For every multigraphon $W$ (see \eqref{def_eq_W_intro}) and multigraph $F \in \M_k$ with adjacency matrix $A \in \A_k$ we define
\begin{equation}
\label{def_grafon_ind_hom_sur}
 t_{=}(F,W):=  t_{=}(A,W):=  \int_{[0,1]^k} \prod_{i\leq j \leq k}
W(x_i,x_j,A(i,j))\, \mathrm{d} x_1\, \mathrm{d} x_2\, \dots\, \mathrm{d}x_k
\end{equation}

We say that $G_n \to W$ if for every $F \in \M$ we have $\lim_{n \to \infty}t_{=}(F,G_n)=t_{=}(F,W)$. 
 By \cite[Theorem 1]{KI_RB} we have that a sequence of multigraphs $\left(G_n\right)_{n=1}^{\infty}$ is convergent then $G_n \to W$ for
some multigraphon $W$.
  The limiting multigraphon of a convergent sequence is not unique,
 but if we define the equivalence relation 
\begin{equation}\label{W_cong}
 W_1 \cong W_2 \quad \iff \quad \forall\, F \in \M :\, t_{=}(F,W_1)=t_{=}(F,W_2)
\end{equation}
 then obviously $G_n \to W_1 \iff G_n \to W_2 $. For other characterisations
of the equivalence relation $\cong$ for graphons, see \cite{lovasz_uniqueness}.

 If $\mathbf{X}_n$ is a random element of $\A_n$ for each $n \in \N$, $\cG_n$ is the isomorphism class of $\mathbf{X}_n$ and
$W$ is a multigraphon, then
 we say that $\mathbf{X}_n \toinp W$ if $\cG_n \toinp W$, see \eqref{convergence_of_random_multigraphons}.

For a multigraphon $W$ and $x \in [0,1]$ we define the \emph{average degree} of $W$ at $x$ and the \emph{edge density} of $W$ by
\begin{align}
\label{def_degree_multtigraphon}
D(W,x) &:= \int_0^1 \sum_{l=0}^{\infty} l\cdot W(x,y,l)\, \mathrm{d}y  \\
\label{def_edge_density_multigraphon}
\rho(W) &:= \int_0^1 \int_0^1 \sum_{l=0}^{\infty} l\cdot W(x,y,l)\, \mathrm{d}y\, \mathrm{d}x
\end{align}
If $\rho(W)<+\infty$ then $D(W,x)<+\infty$ for Lebesgue-almost all $x$.

\medskip
We say that a $[0,1]$-valued random variable $U$ is uniformly distributed on $[0,1]$ (or briefly denote
$U \sim \mathcal{U}[0,1]$) if $\prob{U \leq x}=x$ for all $x \in [0,1]$.

\begin{definition}[$W$-random multigraphons]
\label{def_X_W}
$ $

 Fix $k \in \N$.
 Let $\left(U_i\right)_{i=1}^{k}$  be i.i.d., $U_i \sim \mathcal{U}[0,1]$.
 Given a multigraphon $W$
  we define the $\A_k$-valued random variable
 $\mathbf{X}_W^{[k]} = \left(X_W(i,j)\right)_{i,j=1}^{k}$ as follows: 

Given the background variables 
 $(U_i)_{i=1}^{k}$ the random variables
 $\left(X_W(i,j)\right)_{i \leq j \leq k}$ are conditionally independent and
 $\condprob{ X_W(i,j)=l}{(U_i)_{i=1}^k }=W(U_i,U_j,l)$,  that is
 \begin{equation}\label{X_W_indep_prod_formula}
\forall \, A \in \A_k: \quad
 \condprob{ \mathbf{X}_W^{[k]}=A}{(U_i)_{i=1}^k }:=
 \prod_{i\leq j \leq k} W(U_i,U_j,A(i,j)).
 \end{equation}
 \end{definition}
In plain words: if $i \neq j$ and $U_i=x$, $U_j=y$
 then the number of multiple edges between the vertices labeled by $i$ and $j$
 in $\mathbf{X}_W$ has  distribution $\big( W(x,y,l) \big)_{l=1}^{\infty}$ and
 the number of loop edges at vertex $i$ has distribution
  $\big( W(x,x,2l) \big)_{l=1}^{\infty}$.

For every multigraphon $W$ and  $A \in \A_k$ we have
\begin{equation} \label{homind_grafon_valszamosan}
  t_{=}(A,W) \stackrel{\eqref{def_grafon_ind_hom_sur},\eqref{X_W_indep_prod_formula}}{=} \prob{ \mathbf{X}_W^{[k]}=A }.
\end{equation}

Recalling \eqref{W_cong} it follows that $W_1 \cong W_2$ if and only if  
$\forall\, k \in \N: \; \mathbf{X}_{W_1}^{[k]} \sim \mathbf{X}_{W_2}^{[k]}$, thus the distribution of the
$W$-random multigraphons determine the multigraphon up to  $\cong$ equivalence.
Recalling \eqref{def_degree_multtigraphon} and \eqref{def_edge_density_multigraphon} we have
\begin{equation}\label{degree_W_expect}
D(W,x)  =\condexpect{X_W(1,2)}{U_1=x},  \qquad \quad
\rho(W) =\expect{X_W(1,2)}.
\end{equation}

Note that the weak law of large numbers (heuristically) implies that 
\begin{equation}\label{heu_average_degree}
\frac{1}{n} d( \mathbf{X}_W^{[n]},i) \approx D(W,U_i), \qquad 1 \ll n 
\end{equation}
This relation is the reason why we gave the name \emph{average degree} to $D(W,x)$.

\subsection{Auxiliary stochastic processes}\label{subsection_aux}

In this subsection we recall the definition and some properties of two stochastic processes:
the  M/M/$\infty$-queue and the C.I.R. process.

\medskip

First recall the formulas defining the Poisson, binomial and gamma distributions:
\begin{align}
\label{def_mypoi}
\mypoi(k,\lambda)&:=e^{-\lambda}\frac{\lambda^k}{k!}\\
\label{def_mybin}
\mybin(k,n,p)&:= \binom{n}{k}p^k (1-p)^{n-k}\\
\label{def_mygamma}
\mygamma(x,\alpha,\beta)&:=x^{\alpha-1} \frac{\beta^{\alpha}
e^{-\beta x}}{\Gamma(\alpha)} \ind[x>0]
\end{align}
We say that a nonnegative integer-valued random variable $X$ has Poisson distribution with parameter
 $\lambda$ (or briefly denote $X \sim \text{POI} \left( \lambda \right)$) if $\prob{X=k}=\mypoi(k,\lambda)$ for
all $k \in \N$. We say that a $\{0,1, \dots,n\}$-valued random variable $Y$ has binomial distribution
 with parameters $n$ and $p$ (or briefly denote $Y \sim \text{BIN}(n, p)$)
if $\prob{Y=k}=\mybin(k,n,p)$ for all $k \in \{0,1, \dots,n\}$.
 We say that a nonnegative real-valued random variable $Z$ 
has gamma distribution with parameters $\alpha$ and 
$\beta$ (or briefly denote $Z \sim \text{Gamma}(\alpha, \beta)$) if $\prob{Z \leq z}=\int_0^z \mygamma(x,\alpha,\beta) \, \mathrm{d} x  $.

\medskip

 The M/M/$\infty$-queue with arrival rate $\mu$ and service rate $1$  is an
$\N_0$-valued continuous-time Markov chain $Y_t$, $t \in [0,+\infty)$ with infinitesimal jump rates
 \begin{align}
\label{MMinfty_up}
 \condprob{Y_{t+\mathrm{d}t}=k+1}{Y_{t}=k}& = \mu \,
 \mathrm{d}t + \ordo(\mathrm{d}t) \\
\label{MMinfty_down}
 \condprob{Y_{t+\mathrm{d}t}=k-1}{Y_{t}=k}& = k\, \mathrm{d}t + \ordo(\mathrm{d}t) \\
\label{MMinfty_stay}
\condprob{Y_{t+\mathrm{d}t}=k}{Y_{t}=k}&=1-(\mu+k )\mathrm{d}t+ \ordo(\mathrm{d}t)
\end{align}
Heuristically, $Y_t$ is the length of a queue at time $t$, where customers arrive according to a Poisson process
with rate $\mu$, customers are served parallelly and each customer is served with rate $1$. 
It is well-known (see \cite[Exercise 5.8]{sorbanallas}) that if $Y_0=h \in \N_0$ then
\begin{equation}\label{MMinfty}
\condprob{ Y_t=k}{Y_0=h}=\mysor(t,h,k,\mu):=\sum_{l=0}^k \mybin(l,h, e^{-t}) \cdot \mypoi(k-l, (1-e^{-t}) \mu),
\end{equation}
i.e. $Y_t$ has the same distribution as the sum of two independent random variables with
 $\text{BIN}(h, e^{-t})$ and $\text{POI} \left( (1-e^{-t})\mu \right)$ distributions. From \eqref{MMinfty}
we get that indeed $Y_t \toinp h$ as $t \to 0$ and the stationary distribution of
 the queue is $\text{POI} \left( \mu \right)$:
\begin{equation}\label{queue_t_0_infty}
 \lim_{t \to 0} \mysor(t,h,k, \mu)= \ind[k=h], \qquad \lim_{t \to \infty} \mysor(t,h,k,\mu)=\mypoi(k,  \mu).
\end{equation}

\medskip

 Fix $\kappa, \rho \in (0,+\infty)$. The  Cox$-$Ingersoll$-$Ross (C.I.R.) process is a diffusion process with stochastic differential equation
\begin{equation} \label{CIR_SDE_sketch}
 \mathrm{d}Z_t= \left(\kappa-\frac{\kappa}{\rho}
Z_t \right) \mathrm{d}t + \sqrt{2  Z_t}\, \mathrm{d}B_t,
 \end{equation}
where $B_t$ denotes the standard Brownian motion (for an introduction to SDE, see \cite{revuz_yor}).

Heuristically the SDE \eqref{CIR_SDE_sketch} tells us the mean and variance of small incerements of the continuous-time $\R_+$-valued
Markov process $\left(Z_t\right)_{t \geq 0}$ given the present value of $Z_t$:
\begin{equation}\label{CIR_infinitesimal_E_D}
  \condexpect{  Z_{t+\mathrm{d}t} - z }{Z_t=z} \approx
\left( \kappa - \frac{\kappa}{\rho} z   \right) \mathrm{d}t,
\qquad \condvar{  Z_{t+\mathrm{d}t}-z }{Z_t=z} \approx 2 z \, \mathrm{d}t
\end{equation}

It is well-known (see \cite[Chapter 4.6]{penzugyimatek}) that if we denote
 \[
 \alpha:=\frac{\kappa}{\rho} \qquad \text{ and } \qquad \tau(\alpha,t):=\frac{\alpha}{\exp(\alpha t)-1}
 \]
  and if we start the process $\left(Z_t\right)_{t \geq 0}$ from the initial value $Z_0=z$
 then  $2 (\tau(\alpha,t)+\alpha) \cdot Z_t$ follows a
noncentral chi-square distribution with $2 \kappa$ degrees of freedom and non-centrality parameter
$2z \cdot \tau(\alpha,t)$, thus  we have
 $\condprob{Z_t \leq x}{Z_0=z}=\int_0^x f(t,z,y)\, \mathrm{d}y$ where
\begin{equation}\label{atmenetsurusegfuggveny_CIR}
f(t,z,y)=
\sum_{i=0}^{\infty}
\mypoi(i,z \cdot \tau(\alpha,t) )\mygamma(y,\kappa+i, \tau(\alpha,t)+ \alpha ).
\end{equation}
Note that using \eqref{atmenetsurusegfuggveny_CIR} one can derive
that indeed $Z_t \toinp z$ as $t \to 0$ and the stationary distribution of
 $\left(Z_t\right)_{t \geq 0}$ is $\text{Gamma}(\kappa,\frac{\kappa}{\rho})$:
\begin{equation}\label{cir_0_infty}
\lim_{t \to 0} f(t,z,y)= \delta_{z,y}, \qquad \lim_{t \to \infty}f(t,z,y) = \mygamma(y,\kappa,\frac{\kappa}{\rho}).
\end{equation}

\subsection{Statements of Theorem \ref{theorem_elek_fejlodese} and  Theorem \ref{theorem_fokszamok_fejlodese} }
\label{subsection_statements_of_thms}

In this subsection we state the main results of this paper describing the time evolution of the limiting multigraphons
of a sequence of edge reconnecting models $\mathbf{X}_n(\cdot)$, $n \to \infty$.

In Theorem
\ref{theorem_elek_fejlodese} we precisely formulate  $\mathbf{X}_n(t \cdot n^2) \toinp \breve{W}_t$.

In Theorem
\ref{theorem_fokszamok_fejlodese} we precisely formulate  $\mathbf{X}_n(t \cdot n^3) \toinp \tilde{W}_t$.

Note that in \eqref{intro_conv_theorems}, \eqref{full_characterization} and above  we used the notations
 $\breve{W}_t$ and $\tilde{W}_t$ in order to give the most simple formulations of these results,
 nevertheless our real notations are going to be slightly different.

$ $

Now we describe the evolution of the edge reconnecting model by describing the evolution of the limiting multigraphons.
We consider a sequence of initial multigraphs $\left(G_n\right)_{n=1}^{\infty}$ which converge to a multigraphon
$W$. We assume $\abs{V(G_n)}=n$.
We denote the adjacency matrix of $G_n$ by $B_n \in \A_n$. We assume that the technical condition
 \begin{equation}\label{exponential_moment_initial}
\exists \, \lambda>0,\; C<+\infty\;\; \forall\, n: \quad \frac{1}{\binom{n}{2}} \sum_{i< j \leq n} e^{\lambda B_n(i,j)} \leq C,
\quad \frac{1}{n} \sum_{i=1}^n e^{\lambda B_n(i,i)} \leq C
 \end{equation}
holds.

First we state Theorem \ref{theorem_elek_fejlodese} about the evolution of the edge reconnecting model on the $T=\Ordo(n^2)$ timescale.
In the Introduction this result was referred to as
$\cG_n(t \cdot n^2) \toinp \breve{W}_t$ in order to make the notation as simple as possible.
 In fact we are going to prove 
$\cG_n(t \cdot \frac{\rho(W)}{2} \cdot n^2) \toinp W_t$, thus $W_t=\breve{W}_{\frac{\rho}{2} t}$. The notation
$\breve{W}_t$ will no longer be used.

\begin{theorem}\label{theorem_elek_fejlodese}
Let us fix $\kappa \in (0, +\infty)$.
We consider the edge reconnecting model $\mathbf{X}_n(T)$, $T=0,1,\dots$ on the state space $\A_n^{m(n)}$ and initial state
$\mathbf{X}_n(0)=B_n \in \A_n^{m(n)}$ for $n=1,2,\dots$.
 We assume $B_n \to W$ for some multigraphon $W$  and that \eqref{exponential_moment_initial} holds.

Then for all $t \in [0,+\infty)$ we have
\begin{equation}\label{thm_elek_statement_conv_graphlim}
\mathbf{X}_n \left( \left\lfloor t \cdot \frac{ \rho(W) \cdot n^2}{2}
 \right\rfloor \right ) \toinp W_t \qquad \text{ as } \qquad n \to \infty
 \end{equation}
  where (recall \eqref{MMinfty})
\begin{equation}\label{Wt_edge_evolution}
  W_t(x,y,k) = \left\{
\begin{array}{ll}
\sum_{h=0}^{\infty} W(x,y,h) \mysor( t,h,k, \frac{D(W,x)\cdot D(W,y)}{\rho(W)} )  & \mbox{ if $x \neq y$} \\
\ind[2 \, | \, k] \cdot
\sum_{h=0}^{\infty} W(x,y,h) \mysor( t,\frac{h}{2},\frac{k}{2}, \frac{D(W,x)\cdot D(W,y)}{2\rho(W)} )  & \mbox{ if $x= y$}
\end{array} \right.
\end{equation}
\end{theorem}
 We  prove Theorem \ref{theorem_elek_fejlodese} in Section \ref{section_proof_of_thm_elek}.
Before  stating further theorems, we devote a few paragraphs to the heuristics behind Theorem \ref{theorem_elek_fejlodese}.

In order to give some insight about \eqref{Wt_edge_evolution}, we now give a 
 probabilistic way to generate a random element of $\A_k$ with the same distribution as $\mathbf{X}_{W_t}^{[k]}$ 
(see Definition \ref{def_X_W}):

We first generate $\mathbf{X}_W^{[k]}$ using the background variables $\left(U_i\right)_{i=1}^k$, then
we get $\mathbf{X}_{W_t}^{[k]}$ by letting the entries $\left(\mathbf{X}_{W_t}(i,j)\right)_{i,j=1}^k$ evolve in time:
\begin{itemize}
 \item  if $i <j$, we run an M/M/$\infty$-queue $Y_t$ with initial value $Y_0=\mathbf{X}_W(i,j)$,
 arrival rate $\frac{D(W,U_i)\cdot D(W,U_j)}{\rho(W)}$ and service rate $1$  and let $\mathbf{X}_{W_t}(i,j):=Y_t$
\item  if $i=j$, we do the same thing
with the only exception being that the queue describing the evolution of the number of 
loop edges has arrival rate  $\frac{D(W,U_i)\cdot D(W,U_j)}{2 \rho(W)}$.
\end{itemize}

\medskip

Now we give a  heuristic argument explaining why do  M/M/$\infty$-queues enter the picture:

We look at the evolution of $X_n(T,i,j)$ for some $i \neq j$ (the case of loop edges is analogous).
We denote by $\rd_n(T,i):= \frac{1}{n} d(\mathbf{X}_n(T),i)$.
 From  \eqref{cW_Vold} and \eqref{Vnew}  it follows that
  \begin{multline}\label{heu_edge_up_rate}
\condprob{ X_n(T+1,i,j)=X_n(T,i,j)+1}{\mathbf{X}_n} \approx
\frac{\rd_n(T,i) \cdot n}{2m}\cdot \frac{\rd_n(T,j)\cdot n + \kappa}{2m+n\kappa}+\\
\frac{\rd_n(T,j) \cdot n}{2m}\cdot \frac{\rd_n(T,i)\cdot n +\kappa}{2m+n\kappa} \approx
 \frac{1}{m} \frac{ \rd_n(T,i) \rd_n(T,j)  }{\rho}
\end{multline}
\begin{equation}\label{heu_edge_down_rate}
\condprob{ X_n(T+1,i,j)=X_n(T,i,j)-1}{\mathbf{X}_n} \approx \frac{ 1}{m}X_n(T,i,j)
\end{equation}
In the statement of Theorem \ref{theorem_elek_fejlodese} we used the time scaling 
$T= \lfloor t \cdot \frac{ \rho \cdot n^2}{2}
 \rfloor \approx t \cdot m$, thus if we denote $\mathrm{d}t:=\frac{1}{m}$ then $T+1$ corresponds to $t+\mathrm{d}t$.
 If we define $Y_t:= X(t \cdot m,i,j)$ and compare \eqref{heu_edge_up_rate}, \eqref{heu_edge_down_rate} to
\eqref{MMinfty_up}, \eqref{MMinfty_down} then we see that
 the time evolution of $Y_t$ approximates that of an M/M/$\infty$-queue with arrival rate
$\mu=\frac{ \rd_n(T,i) \rd_n(T,j)  }{\rho}$ and service rate $1$. We will later see that on the
$T \asymp n^2$ timescale $ \rd_n(T,i)$ does not change significantly, so that we have
\[ \rd_n(T,i) \approx \rd_n(0,i) \stackrel{\eqref{heu_average_degree}}{\approx} D(W,U_i).\] 
We note here that the identity $D(W,x)\equiv D(W_t,x)$ can be formally derived from \eqref{Wt_edge_evolution}.

$ $

Now we look at the evolution of the edge reconnecting model on the $T=\Ordo(n^3)$ timescale.
In the Introduction this result was referred to as
$\cG_n(t \cdot n^3) \toinp \tilde{W}_t$ in order to make the notation as simple as possible. 
In fact we are going to prove 
$\cG_n(t \cdot {\rho(W)} \cdot n^3) \toinp \hat{W}_t$, thus $\hat{W}_t=\tilde{W}_{\rho t}$. The notation
$\tilde{W}_t$ will no longer be used.

\begin{theorem} \label{theorem_fokszamok_fejlodese}
Let us fix $\kappa \in (0,+\infty)$.
We consider the edge reconnecting model $\mathbf{X}_n(T)$, $T=0,1,\dots$ on the state space $\A_n^{m(n)}$ and initial state
$\mathbf{X}_n(0)=B_n \in \A_n^{m(n)}$ for $n=1,2,\dots$.
 We assume $B_n \to W$ for some multigraphon $W$  and that \eqref{exponential_moment_initial} holds.

Then for all $t \in (0,+\infty)$ (but not for t=0) we have
\begin{equation}\label{thm_fokok_statement_conv_graphlim}
\mathbf{X}_n \left( \lfloor t \cdot  \rho(W) \cdot n^3
 \rfloor \right) \toinp \hat{W}_t \qquad \text{ as } \qquad n \to \infty
 \end{equation}
  where
 \begin{equation}
  \label{fokszamok_fejlodese_grafon}
\hat{W}_t(x,y,k) =
 \left\{
\begin{array}{ll}
\mypoi(k,\frac{F_t^{-1}(x)F_t^{-1}(y)}{\rho(W)}) & \mbox{ if $x \neq y$}\\
\ind [2|k]\cdot \mypoi \left(\frac{k}{2},\frac{F_t^{-1}(x)F_t^{-1}(y)}{2\rho(W)}\right) & \mbox{ if $x= y$}
\end{array} \right.
 \end{equation}
and $F_t^{-1}$ is the inverse function of $F_t(x)=\int_0^x f(t,y)\, \mathrm{d}y$ where 
(recall \eqref{atmenetsurusegfuggveny_CIR} and note that $f(t,z,y)$ also depends on the parameters $\kappa$ and $\rho$)
\begin{equation}
\label{CIR_densityfunction}
 f(t,y)=\int_0^\infty  f(t,z,y)
\, \mathrm{d}F_0(z),
 \end{equation}
 and
 $F_0(x)= \int_0^1 \ind[D(W,y) \leq x]\, \mathrm{d}y$, $x \in [0,+\infty)$.
\end{theorem}
We prove Theorem \ref{theorem_fokszamok_fejlodese} in Section \ref{section_fokok_fejl}.
Now we devote a few paragraphs to the heuristics behind Theorem \ref{theorem_fokszamok_fejlodese}.
In order to give some insight about \eqref{fokszamok_fejlodese_grafon}, we now give a 
 probabilistic way to generate a random element of $\A_k$ with the same distribution as $\mathbf{X}_{\hat{W}_t}^{[k]}$:

Let us first generate $\mathbf{X}_W^{[k]}$ using the background variables $\left(U_i\right)_{i=1}^k$.
Let us define  $Z_i(0):=D(W,U_i)$. Now $\left(Z_i(0)\right)_{i=1}^k$ are i.i.d. with probability distribution function $F_0$.
We let $\left(Z_i(t)\right)_{t \geq 0}$ evolve in time according to the SDE \eqref{CIR_SDE_sketch}, so that
$\left(Z_i(t)\right)_{i=1}^k$ are i.i.d. with probability distribution function $F_t$. Given the background variables
$\left(Z_i(t)\right)_{i=1}^k$ let $\mathbf{X}_{\hat{W}_t}(i,j) \sim \text{POI}(\frac{Z_i(t)Z_j(t)}{\rho(W)})$ if $i\neq j$,
 and let $\frac{1}{2} \mathbf{X}_{\hat{W}_t}(i,i) \sim \text{POI}(\frac{Z_i(t)Z_i(t)}{2\rho(W)})$.

\medskip
Now we give a  heuristic argument explaining why do C.I.R. processes enter the picture:

Pick $i \in [n]$ and denote by $\rd_n(T):= \frac{1}{n} d(\mathbf{X}_n(T),i)$. It follows from \eqref{Vold} and \eqref{Vnew} that
\begin{gather}\label{sketch_expect_D}
\condexpect{ \rd_n(T+1)-\rd_n(T)}{\mathbf{X}_n(T)}= \frac{ \rd_n(T) +\frac{\kappa}{n}}{2m+n\kappa} - \frac{ \rd_n(T)}{2m}
\approx \frac{1}{2mn} \left( \kappa - \frac{\kappa}{\rho} \rd_n(T)   \right)\\
\label{sketch_var_D}
\condvar{\rd_n(T+1)-\rd_n(T)}{\mathbf{X}_n(T)} \approx
\frac{ \rd_n(T) +\frac{\kappa}{n}}{2mn+n^2\kappa} + \frac{ \rd_n(T)}{2mn} \approx \frac{1}{2mn} 2 \rd_n(T)
\end{gather}
We look at the time evolution of the stochastic process
$Z_t:=\rd_n \left( \lfloor t\cdot 2nm \rfloor \right)$.
In the statement of Theorem \ref{theorem_fokszamok_fejlodese} we used the time scaling 
 $T = \lfloor t \cdot \rho \cdot n^3 \rfloor \approx t \cdot 2mn  $.
 If we let $\mathrm{d}t=\frac{1}{2nm}$ then $T+1$ corresponds to $t + \mathrm{d}t$.
Let $\mathrm{d}  Z_t:= Z_{t+\mathrm{d}t}-Z_t$.
From \eqref{sketch_expect_D} and \eqref{sketch_var_D} we get
\[ \condexpect{ \mathrm{d} Z_t}{Z_t} \approx \left( \kappa - \frac{\kappa}{\rho} Z_t   \right) \mathrm{d}t
\qquad \condvar{ \mathrm{d}  Z_t}{Z_t} \approx 2 Z_t  \mathrm{d}t\]
Thus the process $Z_t$ approximates the solution of the SDE of the C.I.R. process \eqref{CIR_infinitesimal_E_D}.

\subsection{Properties of $W_t$ and $\hat{W}_t$}\label{subsection_properties_W}

In this  subsection we relate some properties of the multigraphons $W_t$ and $\hat{W}_t$ that appear
in Theorem \ref{theorem_elek_fejlodese} and  Theorem \ref{theorem_fokszamok_fejlodese} to the
\emph{configuration model}. But first,  we
show 
\begin{equation}\label{full_characterization_2}
 \lim_{t \to 0_+} W_t \cong W, \qquad
\lim_{t \to \infty} W_t \cong \lim_{t \to 0_+} \hat{W}_t, \qquad
\lim_{t \to \infty} \hat{W}_t \cong \hat{W}_{\infty}
\end{equation}
where $\hat{W}_{\infty}$ is the multigraphon defined by 
\begin{equation}  \label{stac_poi}
\hat{W}_{\infty}(x,y,k) =
 \left\{
\begin{array}{ll}
\mypoi(k,\frac{F^{-1}(x)F^{-1}(y)}{\rho}) & \mbox{ if $x \neq y$}\\
\ind [2|k]\cdot \mypoi \left(\frac{k}{2},\frac{F^{-1}(x)F^{-1}(y)}{2\rho} \right) & \mbox{ if $x= y$}
\end{array} \right.
 \end{equation}
and $F^{-1}$ is the inverse function of $F(x)=\int_0^x \mygamma(y,\kappa,\frac{\kappa}{\rho}) \mathrm{d}y$, see \eqref{def_mygamma}.

In order to make sense of \eqref{full_characterization_2} we define convergence on the space of multigraphons: we say 
that $\lim_{n \to \infty} W_n \cong W$ if $\lim_{n \to \infty} t_=(F,W_n) \to t_=(F,W)$ for all $F \in \M$.
The limit of a convergent sequence is only determined up to the $\cong$ equivalence, see \eqref{W_cong}.

One can see by looking at \eqref{Wt_edge_evolution}, \eqref{MMinfty}, \eqref{def_mybin}, \eqref{def_mypoi} that $W_t$ is a continuous function of $t$.

Similarly, \eqref{fokszamok_fejlodese_grafon}, \eqref{CIR_densityfunction}, \eqref{atmenetsurusegfuggveny_CIR}, 
\eqref{def_mypoi}, \eqref{def_mygamma} imply that $\hat{W}_t$ is a continuous function of $t$. 

If we substitute $t \to 0_+$ into \eqref{Wt_edge_evolution}  we indeed get $W_t \to W$ by \eqref{queue_t_0_infty}.

If we substitute $t \to \infty$ into \eqref{fokszamok_fejlodese_grafon}   we  get $\hat{W}_t \to \hat{W}_{\infty}$ by 
 \eqref{cir_0_infty}.

If we let $t \to \infty$ in \eqref{Wt_edge_evolution} and $t \to 0_+$ in \eqref{fokszamok_fejlodese_grafon}  we get 
\begin{align}\label{semi_stationary_edge_W}
\lim_{t \to \infty} W_t(x,y,k) &\stackrel{\eqref{queue_t_0_infty}}{=}
 \left\{
\begin{array}{ll}
\mypoi(k,\frac{D(W,x)D(W,y)}{\rho(W)}) & \mbox{ if $x \neq y$}  \\
\ind [2|k]\cdot \mypoi \left(\frac{k}{2},\frac{D(W,x)D(W,y)}{2\rho(W)}\right)   & \mbox{ if $x= y$} 
\end{array} \right.
\\
\lim_{t \to 0_+}  \hat{W}_t(x,y,k) &\stackrel{\eqref{cir_0_infty}}{=}
 \left\{
\begin{array}{ll}
\mypoi(k,\frac{F_0^{-1}(x)F_0^{-1}(y)}{\rho(W)}) & \mbox{ if $x \neq y$}\\
\ind [2|k]\cdot \mypoi \left(\frac{k}{2},\frac{F_0^{-1}(x)F_0^{-1}(y)}{2\rho(W)}\right) & \mbox{ if $x= y$}
\end{array} \right.
\end{align}
where $F_0(z)= \int_0^1 \ind[D(W,y) \leq z]\, \mathrm{d}y$ and $F_0^{-1}(x):=\min \{z \, : \, F_0(z)\geq x \}$.
We have $\lim_{t \to \infty} W_t \cong \lim_{t \to 0_{+}} \hat{W}_t$, because the corresponding $W$-random 
multigraphons have the same
distribution, since $D(W,U_i) \sim F_0^{-1}(U_i)$. Thus we have seen that \eqref{full_characterization_2} holds.

\medskip

A well-known way to generate a random multigraph with a prescribed degree sequence  is called the \emph{configuration model}:
 we draw $d(v)$ stubs (half-edges) at each vertex $v$ and then we uniformly choose one from the the set of possible matchings of these stubs. 
 In \cite{RB_SZL} we call such random multigraphs \emph{edge stationary}  and in 
 \cite[Theorem 1]{RB_SZL} we  characterize
the special form of limiting multigraphons that arise as the limit of  edge stationary dense multigraph sequences:
these multigraphons are of form \eqref{stac_poi} where $F$ is a generic probability distribution function on $\R_+$ and
 $F^{-1}(x):=\min \{z \, : \, F(z)\geq x \}$ is the generalized inverse of $F$. The name of edge stationarity comes from the fact that
the space of edge stationary distributions is invariant under
the edge reconnecting dynamics, see \cite[Section 4]{RB_SZL}. The stationary distribution of the edge reconnecting model
 is an example of an edge stationary multigraph, see
\eqref{stac_multigraphon_intro},\eqref{stac_poi}.

The heuristic explanation of the fact that the limiting multigraphon
$\lim_{t \to \infty} W_t$ from \eqref{semi_stationary_edge_W} has the special form that 
appears in  \cite[Theorem 1]{RB_SZL} is as follows:
If $T \approx t \cdot n^2$ where $1 \ll t$  then the degrees of vertices in $\cG_n(0)$ and $\cG_n(T)$ are very close to each other 
(c.f. $D(W,x)\equiv D(W_t,x)$),
 whereas $T$ steps are enough for the model to rearrange and mix the edges, so $\cG_n(T)$ looks edge stationary.
Similarly, for any $t>0$,  $\hat{W}_t$ from \eqref{fokszamok_fejlodese_grafon} also looks edge stationary.

Roughly speaking, if we start the edge reconnecting model from an arbitrary initial multigraph, then we have to run our process for
  $ n^2 \ll T $ steps until $\cG_n(T)$ becomes ``edge stationary'' and run it for 
$ n^3 \ll T$ steps until $\cG_n(T)$  becomes ``stationary''.

\section{Vertex exchangeable random adjacency matrices}\label{section_vert_ex}

In this section we define the notion of vertex exchangeability of random adjacency matrices and recall
 two lemmas from \cite{RB_SZL}: in  Lemma \ref{lemma_homkonv_konv_in_dist}  we relate convergence of dense random multigraphs to
 convergence of the probability measures of the corresponding vertex exchangeable random arrays and in Lemma 
\ref{lemma_uniform_integrabiliy_b} we give sufficient conditions under which convergence of dense random multigraphs
imply convergence of the degree distribution of these graphs, see \eqref{heu_average_degree}.

$ $

Let $\mathbf{X}=\left( X(i,j) \right)_{i,j=1}^n$ denote a random element of $\A_n$.
We say that the distribution $\mathbf{X}$ is \emph{vertex exchangeable} if
  for all permutations $\tau: [n] \to [n]$ 
the $\A_n$-valued random variables
$\left(X(i,j)\right)_{i,j=1}^n$ and $\left(X(\tau(i),\tau(j))\right)_{i,j=1}^n$ have the same distribution:
\begin{equation}\label{finite exchangeability}
\left(X(i,j)\right)_{i,j=1}^n \sim \left(X(\tau(i),\tau(j))\right)_{i,j=1}^n.
\end{equation}
In graph theoretic terms \eqref{finite exchangeability} means that the distribution of the random graph is invariant under the relabeling.
It follows from Definition \ref{def_X_W} that $\mathbf{X}_W^{[k]}$ is vertex exchangeable.

\medskip

In the statements of Theorem \ref{theorem_elek_fejlodese} and Theorem \ref{theorem_fokszamok_fejlodese} the initial state
 of the Markov chain $\mathbf{X}_n(T)=\left( X_n(T,i,j) \right)_{i,j=1}^n$ was the deterministic adjacency matrix $\mathbf{X}_n(0)=B_n$, but if we define
 \begin{equation}\label{permuted_initial_state}
  \hat{X}_n(0,i,j) := B_n\left( \pi(i), \pi(j) \right).
\end{equation}
where $\pi$ denotes a uniformly chosen random permutation of $[n]$ and denote the edge reconnecting
 Markov chain with this initial
 distribution by $\hat{\mathbf{X}}_n(T)$, $T=1,2,\dots$, then
\begin{equation}\label{replace_with_exch_trick}
   \Big( \hat{X}_n(T,i,j) \Big)_{i,j=1}^n \sim \Big( X_n(T,\pi(i),\pi(j)) \Big)_{i,j=1}^n,
\quad   t_{=}(A, \mathbf{X}_n) \sim t_{=}(A, \hat{\mathbf{X}}_n),
\end{equation}
thus we get that the assertion of Theorem \ref{theorem_elek_fejlodese} and 
Theorem \ref{theorem_fokszamok_fejlodese} holds for $\mathbf{X}_n(T)$ if and only if it holds for $\hat{\mathbf{X}}_n(T)$.
 From now on we are going to use this trick to replace $\mathbf{X}_n(T)$ by $\hat{\mathbf{X}}_n(T)$ and assume
 that the distribution 
of $\mathbf{X}_n(T)$ is vertex exchangeable.

\medskip

If $\mathbf{X}_n$ is a random element of $\A_n$ for each $n \in \N$ 
 and $W$ is a multigraphon then we say that $\mathbf{X}_n$ converges in distribution to $\mathbf{X}_W$
as $n \to \infty$ (or briefly denote $\mathbf{X}_n \toind \mathbf{X}_W$) if for all $k \in \N$ we have 
$\mathbf{X}_n^{[k]} \toind \mathbf{X}_W^{[k]}$, i.e.
\begin{equation*}
 \forall \, k \in \N, \; A \in \A_n: \; \; \lim_{n \to \infty} \prob{ \mathbf{X}_n^{[k]}=A } 
= \prob{ \mathbf{X}_W^{[k]}=A }\stackrel{\eqref{homind_grafon_valszamosan}}{=}t_{=}(A,W)
\end{equation*}
Recall that we say that $\mathbf{X}_n \toinp W$ if
\begin{equation*}
 \forall\, k \in \N \;\; \forall \, A \in \A_k   \;\; \forall \, \varepsilon>0:  \;
\lim_{n \to \infty} \prob{ \abs{ t_{=}(A, \mathbf{X}_n) - t_{=}(A, W) }>\varepsilon}=0.
\end{equation*}

We state here  \cite[Lemma 3.1]{RB_SZL} without proof:
\begin{lemma}\label{lemma_homkonv_konv_in_dist}
 Let $\mathbf{X}_n=\left( X_n(i,j) \right)_{i,j=1}^n$ be a random, vertex exchangeable element of  $\A_n$ for all $n \in \N$. 
The following statements are equivalent:
\begin{equation}\label{eq__homkonv_konv_in_dist}
 \mathbf{X}_n \toinp W \quad \iff \quad
 \mathbf{X}_n \toind \mathbf{X}_W.
\end{equation}
\end{lemma}

For a real-valued nonnegative random variable $X$ define
\[\expect{ X ; m}:= \expect{ X \cdot \ind[ X \geq m]}.\]
A sequence of  real-valued nonnegative random variables $\left(X_n\right)_{n=1}^{\infty}$ is uniformly integrable 
(see  \cite[Chapter 13]{williams}) if
\begin{equation*}
\lim_{m \to \infty} \max_{n} \expect{ X_n ; m}=0.
\end{equation*}

We state here a special case of \cite[Lemma 3.2/(ii)]{RB_SZL} without proof:
\begin{lemma} \label{lemma_uniform_integrabiliy_b}
If  $\mathbf{X}_n$ is a random vertex exchangeable element  of $\A_n$ for each $n \in\N$, 
$\mathbf{X}_n \toind \mathbf{X}_W$ holds for some
 multigraphon $W$
and  the sequences \[\left(X_n(1,1)\right)_{n=1}^{\infty} \quad \text{and} \quad
 \left(X_n(1,2)\right)_{n=1}^{\infty}\] are uniformly integrable then 
 for all $k \in \N$ we have
\begin{equation*}
 \left( \mathbf{X}_n^{[k]},
 \left( \frac{1}{n} d(\mathbf{X}_{n},i)\right)_{i=1}^k \right) \toind
 \left( \mathbf{X}_W^{[k]},
 \left( D(W,U_i)\right)_{i=1}^k \right),
\end{equation*}
where $\mathbf{X}_W^{[k]}$ is generated using the background variables $\left(U_i\right)_{i=1}^k$ according to
Definition \ref{def_X_W}.
\end{lemma}

\section{Bounds on multiple edges and degrees}\label{section_technical_bounds}

In this section we state and prove Lemma \ref{lemma_regularity} which, roughly speaking, states that degrees and multiple edges in the
edge reconnecting model remain well-behaved.

\medskip

If we replace the initial matrix $B_n$ with its vertex exchangeable version $\left(X_n(0,i,j)\right)_{i,j=1}^n$ using the trick
\eqref{permuted_initial_state} then the 
technical condition \eqref{exponential_moment_initial} becomes
 \begin{equation}\label{exponential_moment_initial_ketto}
\exists \, \lambda>0,\; C<+\infty\;\; \forall\, n\; \forall\, i,j \in [n] :\; \quad \expect{e^{\lambda X_n(0, i,j)} } \leq C
 \end{equation}
It is easy to see that \eqref{exponential_moment_initial_ketto} implies
that the sequences $\left(X_n(0,1,2)\right)_{n=1}^{\infty}$ and $\left(X_n(0,1,1)\right)_{n=1}^{\infty}$
are uniformly integrable. If we assume $B_n \to W$ and define $\rho:=\rho(W)$ then
 \begin{multline}\label{edge_density_conv}
\lim_{n \to \infty} \frac{2m(n)}{n^2}= \lim_{n \to \infty} \frac{1}{n^2} \sum_{i,j=1}^n X_n(0,i,j)=
\lim_{n \to \infty} \frac{1}{n^2} \sum_{i,j=1}^n \expect{X_n(0,i,j)} \stackrel{\eqref{finite exchangeability}}{=}\\
\lim_{n \to \infty} \left( \frac{n-1}{n} \expect{ X_n(0,1,2)} + \frac{1}{n} \expect{ X_n(0,1,1)} \right)
\stackrel{\eqref{eq__homkonv_konv_in_dist},\eqref{exponential_moment_initial_ketto}}{=} 
\expect{ X_W(1,2)}
 \stackrel{\eqref{degree_W_expect}}{=}\rho
\end{multline}

$ $

Now we state and prove a lemma which says that if the initial state $\mathbf{X}_n(0)$ of the edge reconnecting model is well-behaved 
(i.e. \eqref{exponential_moment_initial_ketto} holds) then model remains well-behaved at later times $T$ as well:
\begin{enumerate}[(i)]
 \item For all $T =\Ordo(n^3)$ the normalized degree  $D(T)=\frac{1}{n} d( \mathbf{X}_n(T),i)$ of a vertex $i \in [n]$
satisfies $D(T)=\Ordo(1)$ uniformly in $n$, a bit more precisely: $\prob{ D(T) \geq z}$ decays exponentially as $z \to \infty$.
\item  For all $T =\Ordo(n^3)$ the number of parallel/loop edges $X_n(T,i,j)$ between vertices $i,j \in [n]$
satisfies $X_n(T,i,j)=\Ordo(1)$ uniformly in $n$, a bit more precisely: $X_n(T,i,j)$ has finite moments.
\item if $T_1 \leq T_2 =\Ordo(n^3)$ and $T_2-T_1 \ll n^3$ then $D(T_1) \approx D(T_2)$, a bit more precisely:
the second moment of $D(T_1) - D(T_2)$ is $\Ordo(n^{-3} (T_2-T_1))$.
\end{enumerate}

 \begin{lemma} \label{lemma_regularity}
Let us fix $\kappa, \rho \in (0,+\infty)$.
We consider the edge reconnecting model $\mathbf{X}_n(T)$, $T=0,1,\dots$ on the state space $\A_n^{m(n)}$ with a vertex exchangeable initial state
$\mathbf{X}_n(0)$ for $n=1,2,\dots$ satisfying $\lim_{n \to \infty} \frac{2m(n)}{n^2}=\rho$ and \eqref{exponential_moment_initial_ketto}
for some
$\lambda< \frac{\kappa}{\rho}$. Then

\begin{enumerate}[(i)]
\item 
 \label{lemma_i_degree_bound} There exists an $n' \in \N$ such that for every $t,z \in [0,+\infty)$ and $n \geq n'$ we have
\begin{equation}\label{exp_lecseng_a_foxam}
\prob{ \max_{0 \leq T \leq 2mnt} \frac{1}{n} d( \mathbf{X}_n(T),i) \geq z}
\leq C\cdot e^{2\lambda \kappa t} \cdot e^{-\lambda z}
\end{equation}
with the $C$ of \eqref{exponential_moment_initial_ketto}.

\item \label{lemma_ii_edge_bound} For every $p>1$ and $t \in [0,+\infty)$ there exists a
$C'=C'(\kappa,\rho,\lambda, C, p,t)$ (where $C$ is the constant from \eqref{exponential_moment_initial_ketto}) such that
for all $n \in \N$, $i,j \in [n]$ and $T \leq 2mnt$ we have 
\begin{equation}\label{moment_p_bound}
 \expect{X_n(T,i,j)^p}\leq C'.
\end{equation}

\item \label{lemma_iii_degree_change}
There exists a constant $C''=C''(\kappa,\rho,\lambda, C, t)$ such
that for all $n \in \N$, all $T_1 \leq T_2 \leq t \cdot n^3$ and $i \in [n]$ we have
\begin{equation}\label{degree_doesnt_change}
\expect{ \left(\frac{1}{n} d( \mathbf{X}_n(T_2),i)-\frac{1}{n} d( \mathbf{X}_n(T_1),i)\right)^2} \leq \frac{C'' \cdot (T_2-T_1)}{n^3}
\end{equation}
\end{enumerate}
\end{lemma}

$ $

\begin{proof}[Proof of Lemma  \ref{lemma_regularity} \eqref{lemma_i_degree_bound}]
Fix $i \in [n]$ and denote
\[d(T):=d(\mathbf{X}_n(T),i), \qquad   D(T):=\frac{1}{n} d( \mathbf{X}_n(T),i).\]
 Denote by
$\left(\cF_T\right)_{0\leq T}$  the natural filtration generated by the process.

If $a(T):=\condexpect{e^{\lambda \cdot (D(T+1)-D(T))}-1}{\cF_T}$ then
$M(T):=e^{\lambda D(T)} \prod_{l=0}^{T-1}(1+a(l))^{-1}$ is a nonnegative martingale.
By Doob's submartingale inequality we have
\begin{equation}\label{szubmartingal_egyenlotlenseg}
\prob{ \max_{0 \leq T \leq T'} M(T) \geq x} \leq \frac{ \expect{e^{\lambda D(0)}}}{x}
\end{equation}
\begin{equation}\label{exponencialis_konvex}
\max_{0 \leq T \leq T'} M(T) < x \quad \implies \quad \forall\, T \leq T': \;
e^{\lambda D(T)} \leq x \exp\left( \sum_{l=0}^{T-1} a(l) \right)
\end{equation}
Now we give an upper bound on $a(T)$. Using
\begin{equation}\label{degree_ind_evolution}
D(T+1)=D(T)+\frac{1}{n} \ind[\cV_{new}(T)=i]-\frac{1}{n}\ind[\cV_{old}(T)=i],
\end{equation}
 \eqref{Vold}, \eqref{Vnew}
 and the fact that $\cV_{new}(T)$
and $\cV_{old}(T)$ are conditionally independent given $\cF_T$ we get
\begin{multline}\label{a_T_becsles1}
a(T)=\left(1+ \frac{d(T)+\kappa}{2m+n\kappa} (e^{\frac{\lambda}{n}}-1)\right)
\left( 1+\frac{d(T)}{2m}(e^{-\frac{\lambda}{n}}-1)\right)-1 \leq\\
\frac{d(T)+\kappa}{2m+n\kappa}\left( \frac{\lambda}{n} +
\frac12 e^{\frac{\lambda}{n}} \frac{\lambda^2}{n^2}\right)+
\frac{d(T)}{2m}\left( -\frac{\lambda}{n} +
\frac12 e^{\frac{\lambda}{n}} \frac{\lambda^2}{n^2}\right)=\\
\frac{\lambda}{n} \frac{1}{4m^2+n 2m \kappa}
\left( d(T) \cdot \left( e^{\frac{\lambda}{n}} \lambda \cdot
\left( \frac{2m}{n}+\frac12 \kappa \right) -n\kappa\right)
+2m\kappa \cdot \left( 1+ \frac12 e^{\frac{\lambda}{n}} \frac{\lambda}{n}\right) \right)
\end{multline}
Now $\lambda<\frac{\kappa}{\rho}$ and $\lim_{n \to \infty} \frac{2m(n)}{n^2}=\rho$, thus if $n$ is big enough
 then $\lambda<e^{-\frac{\lambda}{n}} \cdot \frac{\kappa}{\rho +\frac12 \frac{\kappa}{n}}$, which implies that the coefficient of $d(T)$ is negative in 
the right hand side of \eqref{a_T_becsles1}, thus
\begin{equation}\label{a_T_becsles2}
a(T) \leq \frac{1}{2mn} \lambda \kappa \frac{1+\frac12 \frac{\lambda}{n} e^{\frac{\lambda}{n}}}{1+ \frac{n \kappa}{2m}} \leq
 \frac{1}{2mn} 2 \lambda \kappa.
\end{equation}
From \eqref{szubmartingal_egyenlotlenseg}, \eqref{exponencialis_konvex} and \eqref{a_T_becsles2} it follows that
\begin{equation*}
\prob{\max_{0 \leq T \leq 2mnt} e^{\lambda D(T)} \geq
x \exp(2 \lambda \kappa t)} \leq \frac{ \expect{e^{\lambda D(0)}}}{x}
\end{equation*}
Substituting $x=\exp(-2\kappa \lambda t)\exp(\lambda z)$ and using
\begin{equation*}
\expect{ e^{\lambda D(0)}}=
\expect{ \exp\left( \frac{1}{n} \sum_{j=1}^n \lambda X(0,i,j)\right)} \leq
\expect{ \frac{1}{n} \sum_{j=1}^n \exp(\lambda X(0,i,j))} \stackrel{ \eqref{exponential_moment_initial_ketto} }{\leq}  C
\end{equation*}
we arrive at \eqref{exp_lecseng_a_foxam}.
\end{proof}

$ $

\begin{proof}[Proof of Lemma  \ref{lemma_regularity} \eqref{lemma_ii_edge_bound}]
Fix $n$ and $i, j \in [n]$. We only prove the statement of the lemma if $i \neq j$, the proof of the diagonal case is similar.
   Denote by
\[ X(T):=X_n(T,i,j), \qquad d(T,i):= d( \mathbf{X}_n(T),i), \qquad D(T,i)=\frac1n d(T,i).\]
 Using \eqref{edge_evolution_ind} we get
\begin{multline}\label{el_behuzas_valsege}
\condprob{X(T+1)=X(T)+1}{\cF_T} =\\
 \frac{d(T,i)+\kappa}{2m+n \kappa}
\left( \frac{d(T,j)}{2m} \left( 1- \frac{X(T)}{d(T,j)} \right) \right) +
\frac{d(T,j)+\kappa}{2m+n \kappa} \left( \frac{d(T,i)}{2m}
\left( 1- \frac{X(T)}{d(T,i)} \right) \right)
\end{multline}
\begin{equation}\label{el_torles_valsege}
\condprob{X(T+1)=X(T)-1}{\cF_T}=\frac{X(T)}{2m}
\left(1- \frac{d(T,i)+\kappa}{2m+n \kappa} \right) +
\frac{X(T)}{2m} \left(1-\frac{d(T,j)+\kappa}{2m+n \kappa} \right)
\end{equation}
From this it is straightforward to derive
\begin{multline*}
\condexpect{e^{\lambda X(T+1)} - e^{\lambda X(T)}}{\cF_T} \leq \\
e^{\lambda X(T)} \left( (e^{\lambda}-1)
\left( \frac{d(T,i)+\kappa}{2m+n \kappa} \frac{d(T,j)}{2m}+
\frac{d(T,j)+\kappa}{2m+n \kappa} \frac{ d(T,i)}{2m} \right)+
(e^{-\lambda}-1) \frac{X(T)}{m}\right)
\end{multline*}
Define the stopping time
\[ \tau_y:= \min \left\{ T \, : \, \frac{d(T,i)+\kappa}{2m+n \kappa} \frac{d(T,j)}{2m}+
\frac{d(T,j)+\kappa}{2m+n \kappa} \frac{ d(T,i)}{2m} > \frac{y}{m} \right\} \]
and $X_y(T):=X(T) \ind [ \tau_y > T ]$. Now we prove that for all $T \in \N$
\begin{equation}\label{we_prove_18}
\expect{ e^{\lambda X_y(T)}} \leq
\max \left\{ C \, , \,  \exp( y \lambda e^{\lambda})
\left( 1+ \frac{(e^{\lambda}-1)y}{m} \right) \right\}.
\end{equation}

It is straightforward to check that
\begin{equation}\label{condexpect_14}
\condexpect{e^{\lambda X_y(T+1)} - e^{\lambda X_y(T)}}{\cF_T} \leq
e^{\lambda X_y(T)}\left( (e^{\lambda}-1) \frac{y}{m} + (e^{-\lambda}-1) \frac{X_y(T)}{m} \right).
\end{equation}
If we denote $E(T):= \expect{e^{\lambda X_y(T)}}$, take the expectation of \eqref{condexpect_14} and use Jensen's inequality then we get
\begin{equation}\label{expect_jensen_E}
E(T+1)-E(T) \leq \frac{E(T)}{m}
\left( (e^{\lambda}-1)y+(e^{-\lambda}-1)\frac{ \log(E(T))}{\lambda} \right).
\end{equation}
We prove \eqref{we_prove_18} by induction. For $T=0$ we use
\eqref{exponential_moment_initial_ketto}.
If $E(T)>\exp( y \lambda e^{\lambda})$, then by \eqref{expect_jensen_E}
 $E(T+1)<E(T)$ and if $E(T) \leq \exp( y \lambda e^{\lambda})$ then
 \[
 E(T+1) \leq E(T)+\exp( y \lambda e^{\lambda})\frac{(e^{\lambda}-1)y}{m}
 \leq \exp( y \lambda e^{\lambda})
\left( 1+ \frac{(e^{\lambda}-1)y}{m} \right).
\]
Having established \eqref{we_prove_18} we prove \eqref{moment_p_bound} by showing that
\[\expect{X(T)^p}\leq 1+ \int_1^{\infty} \prob{ X(T)^p \geq x} \, \mathrm{d}x<+\infty.\]
\begin{multline*}
\prob{X(T)^p \geq x}
\leq \prob{ X_y(T)^p \geq x}+\prob{ X(T) \neq X_y(T)} \leq
\frac{ \expect{e^{\lambda X_y(T)}}}{\exp(\lambda x^{1/p})}+\prob{ \tau_y>T}
\stackrel{\eqref{we_prove_18} }{ \leq }\\
\frac{C_1 \exp(C_2 y)}{\exp(\lambda x^{1/p})}+
\prob{\max_{T \leq 2nmt} D(T,i)D(T,j)> C_3 y} \stackrel{ \eqref{exp_lecseng_a_foxam} }{\leq} C_1 \exp( C_2 y -\lambda x^{1/p})+C_4 e^{-C_5 \sqrt{y}}
\end{multline*}
Now choosing $y=x^{1/{2p}}$ we indeed get $\int_1^{\infty} \prob{ X(T)^p \geq x} \, \mathrm{d}x<+\infty$.
\end{proof}

$ $

\begin{proof}[Proof of Lemma \ref{lemma_regularity} \eqref{lemma_iii_degree_change}]
Fix $i \in [n]$.
We use the notation $D(T)=\frac{1}{n} d( \mathbf{X}_n(T),i)$. We say that
$a_n =\Ordo(b_n)$ if there exists a constant $c$ depending only on
 $\kappa$,$\rho$,$\lambda$,$C$ and $t$ such that $a_n \leq c \cdot b_n$ for all $n \in \N$.
 It follows from  \eqref{exp_lecseng_a_foxam} that
\begin{equation}\label{ordoegy}
 \forall \, T \leq t \cdot n^3: \; \;
 \expect{D(T)}=\Ordo(1) \qquad
 \forall \, T, T' \leq t \cdot n^3: \; \;
 \expect{D(T)D(T')}=\Ordo(1).
\end{equation}
Using
\eqref{degree_ind_evolution}, \eqref{Vold}, \eqref{Vnew}
 and the fact that $\cV_{new}(T)$
and $\cV_{old}(T)$ are conditionally independent given $\cF_T$ we get
\begin{multline*}
\condexpect{ D(T+1)-D(T)}{\cF_T}= \frac{ D(T) +\frac{\kappa}{n}}{2m+n\kappa} - \frac{ D(T)}{2m}= \frac{\kappa}{2mn+n^2\kappa} -\frac{n \kappa D(T)}{4m^2+2mn \kappa},
\end{multline*}
\begin{multline*}
\condexpect{(D(T+1)-D(T))^2}{\cF_T} =
\frac{1}{n^2}\left( \frac{n D(T) + \kappa}{2m+n \kappa}+\frac{n D(T)}{2m}
-2 \frac{n D(T) + \kappa}{2m+n \kappa}\frac{n D(T)}{2m} \right).
\end{multline*}

We prove  \eqref{degree_doesnt_change} by induction on $T_2-T_1$.
\begin{multline*}
\expect{(D(T_2+1)-D(T_1))^2}= \expect{(D(T_2)-D(T_1))^2}+\\
2 \expect{ \condexpect{ D(T_2+1)-D(T_2)}{\cF_{T_2}}(D(T_2)-D(T_1))} +
\expect{(D(T_2+1)-D(T_2))^2} \stackrel{\eqref{ordoegy}}{=}\\
 \expect{(D(T_2)-D(T_1))^2}+
\Ordo\left(\frac{1}{n^3}\right)
\end{multline*}
\end{proof}

$ $

We state a lemma about the speed of convergence of the M/M/$\infty$-queue to its stationary distribution.
\begin{lemma}\label{lemma_kicsi}
Let $Y_t$ be an
$\N_0$-valued continuous-time Markov chain with infinitesimal jump rates \eqref{MMinfty_up},\eqref{MMinfty_down},\eqref{MMinfty_stay} and initial state
 $h \in \N_0$.
Then for all $t \geq 0$ and $l \in \N_0$ we have
\begin{equation}\label{lemma_kicsi_eq}
\abs{ \prob{ Y_t = l} -\lim_{s \to \infty} \prob{ Y_s = l}} \leq e^{-t}\cdot (h+ \mu)
\end{equation}
\end{lemma}

\begin{proof}[Proof of Lemma \ref{lemma_kicsi}]
According to \eqref{queue_t_0_infty} $Y_s \toind \text{POI}(\mu)$ as $s \to \infty$.
Let
\[ Y_t^{bin} \sim \text{BIN}(h,e^{-t}), \quad  Y_t^{poi} \sim \text{POI}((1-e^{-t})\mu), \quad Y_t^{\infty} \sim \text{POI}(e^{-t} \mu)\] be 
mutually independent random variables.

By \eqref{MMinfty} we have
 $Y_t^{bin}+Y_t^{poi} \sim Y_t$ and $Y_t^{poi}+Y_t^{\infty} \sim \text{POI}(\mu)$.
\begin{multline*}
\abs{ \prob{ Y_t = l} -\lim_{s \to \infty} \prob{ Y_s = l}} =
\abs{ \prob{ Y_t^{bin}+Y_t^{poi} = l} - \prob{ Y_t^{poi}+Y_t^{\infty} = l}} \leq \\
\prob{  Y_t^{bin}+Y_t^{poi} \neq Y_t^{poi}+Y_t^{\infty}} \leq
\prob{ Y_t^{bin} \neq 0}+ \prob { Y_t^{\infty} \neq 0} =\\
1-(1-e^{-t})^h +(1-\exp(-e^{-t} \mu)) \leq e^{-t}\cdot (h+ \mu)
\end{multline*}

\end{proof}

\section{Proof of Theorem \ref{theorem_elek_fejlodese}}
\label{section_proof_of_thm_elek}

In this section we prove Theorem \ref{theorem_elek_fejlodese}
 by coupling the evolution of multiple edges between the vertices
$1 \leq i \leq j \leq k$ to $\binom{k+1}{2}$ independent M/M/$\infty$-queues.

$ $

Given  a random element $\mathbf{X}$  of $\A_k$
 we define the modified adjacency matrix $\mathbf{X}^*$ in the following way:  let
$X^*(i,j):=X(i,j)$ if $i \neq j$ and $X^*(i,i):=\frac12 X(i,i)$.

We assume that the distribution of $\mathbf{X}_n(T)$ is vertex exchangeable
(see the paragraph after \eqref{replace_with_exch_trick}).
We are going to prove \eqref{thm_elek_statement_conv_graphlim}
  using Lemma \ref{lemma_homkonv_konv_in_dist}: we only need to show that for all $k \in \N$ and
 $t \geq 0$ we have
 \begin{equation}\label{thm_edges_formula_convindist}
  \mathbf{X}_n^{[k]} \left( \lfloor t \cdot \frac{ \rho(W) \cdot n^2}{2}
 \rfloor \right)  \toind \mathbf{X}_{W_t}^{[k]} \quad \text{ as } \quad n \to \infty.
 \end{equation}
 Note that the evolution of
 $\left(\mathbf{X}_n^{[k]}(T), \left( d(\mathbf{X}_n(T),i)\right)_{i=1}^k \right)$
 is itself a Markov chain under the edge reconnecting dynamics.

We are going to prove \eqref{thm_edges_formula_convindist} by coupling the $\A_k$-valued discrete-time process
$\mathbf{X}_n^{[k]}(T)$
 to an $\A_k$-valued continuous-time process $\mathbf{Y}_n^{[k]}(t)$
 which we define now:
\begin{itemize}
\item The initial states are the same:
 $\forall \, i,j \in [k]:\; Y_n(0,i,j)=X_n(0,i,j)$.
\item Given $\mathbf{Y}_n^{[k]}(0)=\mathbf{X}_n^{[k]}(0)$, the evolution
 of $Y_n(t,i,j)$ is a continuous-time Markov process for each $i,j \in [k]$,
 the entries $\left(Y_n(t,i,j)\right)_{i\leq j \leq k}$ evolve independently and
 $Y_n(t,i,j)\equiv Y_n(t,j,i)$, thus $\mathbf{Y}_n^{[k]}(t)$ is a random element of $\A_k$.
\item The process $Y^*_n(t,i,j)$ is an M/M/$\infty$-queue (see \eqref{MMinfty_up}, \eqref{MMinfty_down}, \eqref{MMinfty_stay})
  with service rate $1$ and arrival rate
\begin{equation}\label{mminfty_rate_in_Y_i_j}
 \mu=\mu_{i,j}:=\frac{ d(\mathbf{X}_n(0),i)d(\mathbf{X}_n(0),j)}{2m(n)\cdot (1+\ind[i=j]) }.
\end{equation}
\end{itemize}

Now we show that for all $t \geq 0$
 \begin{equation}\label{Y_tablazat_konv_eo}
  \mathbf{Y}_n^{[k]}(t)  \toind \mathbf{X}_{W_t}^{[k]} \quad \text{ as } \quad n \to \infty.
 \end{equation}
From the assumptions $\mathbf{X}_n(0) \toinp W$, \eqref{exponential_moment_initial_ketto} and Lemma
 \ref{lemma_uniform_integrabiliy_b}  it follows that
\begin{equation*}
 \left( \mathbf{Y}_n^{[k]}(0),
 \left( \frac{1}{n} d(\mathbf{X}_n(0),i)\right)_{i=1}^k \right) \toind
 \left( \mathbf{X}_W^{[k]},
 \left( D(W,U_i)\right)_{i=1}^k \right).
\end{equation*}
Now \eqref{Y_tablazat_konv_eo}  easily follows from this, \eqref{MMinfty},
\eqref{edge_density_conv}, Definition \ref{def_X_W}
and \eqref{Wt_edge_evolution}.

 Denote by $\rd_n(T,i):=\frac{1}{n}d(\mathbf{X}_n(T),i)$.
We are going to construct a coupling (joint realization on the same probability space)
  of the discrete time $\A_k$-valued Markov chains  $\mathbf{X}_n^{[k]}(T)$ and
 $\mathbf{Y}_n^{[k]}\left(\frac{T}{m(n)}\right) $ for $T=0,1,\dots$ such that for any $\nu<\frac{5}{2}$ we have
\begin{equation}\label{szoros_a_csat}
\lim_{n \to \infty}
\prob{ \forall \, T \leq n^{\nu}: \quad  \mathbf{X}_n^{[k]}(T) =
 \mathbf{Y}_n^{[k]} \left( \frac{T}{m(n)} \right) } =1.
\end{equation}
Before proving \eqref{szoros_a_csat} we first assume that it holds and deduce  Theorem \ref{theorem_elek_fejlodese} from it:

Fix $t \in (0,+\infty)$.
If $2<\nu$ and $n$ is large enough then
$2 t \cdot m(n)<n^{\nu}$.
It is easy to see that  \eqref{Y_tablazat_konv_eo}, \eqref{szoros_a_csat}
 and $\lim_{n \to \infty} \frac{2m(n)}{n^2}=\rho(W)$ together imply
\eqref{thm_edges_formula_convindist}.

\medskip

Now we start proving \eqref{szoros_a_csat}.

For $i,j \in [k]$ we define the matrix $E_{i,j} \in \A_k$ by
\[ E_{i,j}(i',j'):= \ind[i=i', j=j']+\ind[i=j', j=i'] \]

Fix $n \in \N$. We introduce the events
\begin{align*}
E_X^{\pm}(T,i,j)&:= \left\{  \mathbf{X}_n^{[k]}(T+1) = \mathbf{X}_n^{[k]}(T)
\pm E_{i,j}  \right\} \\
E_Y^{\pm}(T,i,j)&:= \left\{  Y_n^{[k]} \left( \frac{T+1}{m} \right) =
 Y_n^{[k]} \left(\frac{T}{m} \right)  \pm E_{i,j}  \right\} \\
E_X(T,\emptyset)&:= \left\{ \mathbf{X}_n^{[k]}(T+1) = \mathbf{X}_n^{[k]}(T)
  \right\} \\
E_Y(T,\emptyset)&:= \left\{  Y_n^{[k]} \left( \frac{T+1}{m} \right) =
 Y_n^{[k]} \left( \frac{T}{m} \right)  \right\}
\end{align*}
It is straightforward to derive from \eqref{MMinfty}  that there is an absolute constant $\hat{C}$ such that if we
define
\begin{equation}\label{Err_Y_defi}
\text{Err}_Y(T):=\frac{\hat{C}}{m^2}
\left(1+\sum_{i,j=1}^k Y_n \left( \frac{T}{m},i,j \right)+\mu_{i,j}\right)^2
\end{equation}
then
\begin{align}
\label{EYplus_prob}
\abs{\condprob{ E_Y^{+}(T,i,j)}{\cF_T} - \frac{\mu_{i,j}}{m} } &\leq \text{Err}_Y(T) \\
\label{EYminus_prob}
\abs{ \condprob{E_Y^{-}(T,i,j)}{\cF_T}-\frac{ Y^* \left( \frac{T}{m},i,j \right)}{m  }} &\leq \text{Err}_Y(T) \\
\label{EY_empty_prob}
\abs{ \condprob{ E_Y(T,\emptyset)}{\cF_T} - 1+
\frac{ \sum_{i\leq j \leq k}  Y^* \left( \frac{T}{m},i,j \right) + \mu_{i,j}  }{m}}  &\leq \text{Err}_Y(T)
\end{align}
From the definition of the edge reconnecting model it follows (similarly to \eqref{el_behuzas_valsege} and
\eqref{el_torles_valsege}) that there
is a constant $\tilde{C}$ depending only on $\kappa$ and $\rho$ such that if we define
\begin{multline}\label{Err_X_defi}
\text{Err}_X(T):=
 \frac{\tilde{C}}{n^3}
\left(1+\sum_{i,j=1}^k X_n(T,i,j)+ \sum_{i=1}^k \rd_n(T,i)  \right)^2
+\\
\sum_{i,j=1}^k
\frac{1}{m}\abs{\frac{ d(\mathbf{X}_n(T),i)d(\mathbf{X}_n(T),j)}{ 2m \cdot (1+\ind[i=j]) }-
\mu_{i,j}}
\end{multline}
then
\begin{align}\label{EXplus_prob}
\abs{\condprob{ E_X^{+}(T,i,j)}{\cF_T} - \frac{\mu_{i,j}}{m} } &\leq \text{Err}_X(T) \\
\label{EXminus_prob}
\abs{ \condprob{E_X^{-}(T,i,j)}{\cF_T}-\frac{ X^*(T,i,j)}{m }} &\leq \text{Err}_X(T) \\
\label{EXempty_prob}
\abs{ \condprob{ E_X(T,\emptyset)}{\cF_T} - 1+
\frac{ \sum_{i\leq j \leq k}  X^*(T,i,j) + \mu_{i,j}  }{m}}  &\leq \text{Err}_X(T)
\end{align}
For any joint realization (coupling) of the discrete time processes
$\mathbf{X}_n^{[k]}(T)$ and $\mathbf{Y}_n^{[k]} \left( \frac{T}{m} \right)$, $T=0,1,\dots$ define $E(T)$ to be the
event  that the  $\A_k$-valued increment from $T$ to $T+1$ of these two  $\A_k$-valued processes is the same:
\[ E(T):=\left\{
 \left( E_X(T,\emptyset) \cap E_Y(T,\emptyset) \right)
\cup
  \bigcup_{ \epsilon \in \{ +, - \}} \; \bigcup_{i \leq j \leq k}
  \left( E_X^{\epsilon}(T,i,j) \cap E_Y^{\epsilon}(T,i,j) \right) \right\}. \]
For any coupling the inclusion
\begin{equation}\label{coupling_inclusion}
\left\{ \mathbf{X}_n^{[k]}(T)=\mathbf{Y}_n^{[k]} \left( \frac{T}{m} \right) \right\} \cap E(T)
 \subseteq
  \left\{ \mathbf{X}_n^{[k]}(T+1)=\mathbf{Y}_n^{[k]} \left( \frac{T+1}{m} \right) \right\}
\end{equation}
holds. Let 
\begin{equation}\label{err_T}
\text{Err}(T):=2k^2(\text{Err}_X(T)+\text{Err}_Y(T)).
\end{equation}

Now if we compare  \eqref{EYplus_prob} to \eqref{EXplus_prob}, \eqref{EYminus_prob} to \eqref{EXminus_prob} and
 \eqref{EY_empty_prob} to \eqref{EXempty_prob},
 it easily follows  that there exists a coupling for which
 \begin{equation*}
 \condprob{E(T)}{\cF_T} \geq \ind [  \mathbf{X}_n^{[k]}(T)=\mathbf{Y}_n^{[k]} \left(\frac{T}{m} \right) ]\cdot
 \left( 1-\text{Err}(T) \right)
  \end{equation*}
Putting this inequality together with \eqref{coupling_inclusion}, multiplying both sides by
\[\ind[ \forall \, T' \leq T-1: \;\;  \mathbf{X}_n^{[k]}(T')=\mathbf{Y}_n^{[k]} \left(\frac{T'}{m} \right)]\] and taking
the expectation of both sides of the inequality
 we get
\begin{multline*}
\prob{ \forall \, T' \leq T+1: \;\;  \mathbf{X}_n^{[k]}(T')=\mathbf{Y}_n^{[k]} \left(\frac{T'}{m} \right)}
 \geq \\
\prob{ \forall \, T' \leq T: \;\;  \mathbf{X}_n^{[k]}(T')=\mathbf{Y}_n^{[k]} \left(\frac{T'}{m} \right) }
 -\expect{\text{Err}(T)}.
\end{multline*}
Thus in order to prove \eqref{szoros_a_csat} we only need to show
\begin{equation}\label{sum_kicsi_monor}
 \lim_{n \to \infty} \sum_{T=0}^{n^{\nu}} \expect{\text{Err}(T)} =0.
 \end{equation}

\medskip

In the remaning part of this section we prove \eqref{sum_kicsi_monor}.

First we show that if
 $T \leq n^{\nu}$ then 
\begin{equation}\label{bound_X}
\expect{ \text{Err}_X(T)}=\Ordo \left( n^{-5/2} \right).
\end{equation}
Since $\nu<\frac{5}{2}<3$, we have $n^{\nu} \leq nm$ if $n$ is large enough, thus
\begin{equation*}
 \expect{\rd_n(T,i)^2}=\Ordo(1), \qquad \expect{X_n(T,i,j)^2}=\Ordo(1)
\end{equation*}
follow from  Lemma \ref{lemma_regularity} \eqref{lemma_i_degree_bound} and
   and  Lemma \ref{lemma_regularity} \eqref{lemma_ii_edge_bound}, respectively.
\begin{multline}\label{bound_on_rate_evolution}
\expect{\frac{1}{m}\abs{\frac{ d(\mathbf{X}_n(T),i)d(\mathbf{X}_n(T),j)}{ 2m \cdot (1+\ind[i=j]) }-
\mu_{i,j}}}\stackrel{\eqref{mminfty_rate_in_Y_i_j}}{=}\\
\Ordo \left( \frac{1}{n^2} \expect{\abs{ \rd_n(T,i) \rd_n(T,j)-\rd_n(0,i)\rd_n(0,j)}} \right)=\\
\frac{1}{n^2} \Ordo \left( \expect{ \abs{ \rd_n(T,i) - \rd_n(0,i)}\cdot \rd_n(T,j)} +
\expect{ \abs{ \rd_n(T,j) - \rd_n(0,j)}\cdot \rd_n(0,i)} \right)\stackrel{(*)}{=} \\
\frac{1}{n^2} \Ordo \left( \sqrt{ \expect{ (\rd_n(T,i)-\rd_n(0,i))^2}}\sqrt{\expect{ \rd_n(T,j)^2}}+ \right. \\
\left. \sqrt{ \expect{ (\rd_n(T,j)-\rd_n(0,j))^2}}\sqrt{\expect{ \rd_n(0,i)^2}} \right)
\stackrel{\eqref{degree_doesnt_change}}{=}\\ \Ordo \left(\frac{1}{n^2}\sqrt{\frac{n^2}{n^3}}\right)\Ordo(1) =
\Ordo \left( n^{-5/2} \right)
\end{multline}
The equation marked by $(*)$ follows from the Cauchy-Schwartz inequality.

Taking the expectation of \eqref{Err_X_defi} and using \eqref{bound_X}, \eqref{bound_on_rate_evolution}
 we indeed get \eqref{bound_X}.

\medskip

Now we show that if
 $T \leq n^{\nu}$ then 
\begin{equation}\label{bound_Y}
 \expect{ \text{Err}_Y(T)}=\Ordo \left( n^{-4} \right).
\end{equation}

 The proof of $\expect{Y_n \left(\frac{T}{m},i,j \right)^2}=\Ordo(1)$ 
is similar to that of Lemma \ref{lemma_regularity} \eqref{lemma_ii_edge_bound} and we omit it,
 $\expect{ \mu_{i,j}^2} =\Ordo(1)$ follows from Lemma \ref{lemma_regularity} \eqref{lemma_i_degree_bound}.
 Taking the expectation of
 \eqref{Err_Y_defi} we get \eqref{bound_Y}.

\medskip

 Now if we substitute \eqref{bound_Y} and
\eqref{bound_X} into \eqref{err_T} 
we get $\expect{\text{Err}(T)}=\Ordo \left( n^{-5/2} \right)$ from which \eqref{sum_kicsi_monor}  follows using $\nu<\frac{5}{2}$.

\section{Proof of Theorem \ref{theorem_fokszamok_fejlodese}}\label{section_fokok_fejl}
In this section we prove Theorem \ref{theorem_fokszamok_fejlodese}
 in two stages:

In Subsection \ref{subsection_elatkotos_proof_of_degrees_degrees} we prove that
the joint evolution of the (normed, rescaled) degrees of the vertices $1,2,\dots,k$ behave like independent C.I.R. processes if $1 \ll n$. Given this result we prove (using the results of Section \ref{section_proof_of_thm_elek})
that after $n^2 \ll T$ steps the state of the edge reconnecting model is essentially edge stationary
in
Subsection \ref{subsection_elatkotos_proof_of_edgestationarity}.

\subsection{ Evolution of degrees}
\label{subsection_elatkotos_proof_of_degrees_degrees}

\begin{lemma}\label{lemma_degrees_converge_precise}
Let us fix $\kappa \in (0,+\infty)$.
We consider the edge reconnecting model $\mathbf{X}_n(T)$, $T=0,1,\dots$ on the state space $\A_n^{m(n)}$ with a vertex exchangeable initial state
$\mathbf{X}_n(0)$ for $n=1,2,\dots$ satisfying and \eqref{exponential_moment_initial_ketto}.
 We  assume $\mathbf{X}_n(0) \toinp W$ for some multigraphon $W$.

Then for all $t \in [0, +\infty)$ and $k \in \N$ we have
\begin{equation}\label{degrees_convindist_iid_CIR}
 \left( \rd_n \left( \lfloor t \cdot  \rho(W) \cdot n^3
 \rfloor,i \right) \right)_{i\in[k]} \toind \left( Z_{t,i}  \right)_{i \in [k]}  \quad \text{ as } n \to \infty
 \end{equation}
 where $\left( Z_{t,i}  \right)_{i \in [k]}$ are i.i.d. with distribution function $F_t(x)=\int_0^x f(t,y)\, \mathrm{d}y$ where
 $f(t,x)$ is defined by \eqref{CIR_densityfunction}.
\end{lemma}

In order to prove this lemma, we are going to apply a special case of \cite[Corollary 2.2]{ispany}, which we
reformulate to fit our needs and notation:

\begin{theorem}\label{ispany_theorem}
Let $\beta: \R \to \R$ and $\gamma: \R \to \R$ be continuous
functions. Assume that the stochastic differential equation \be
\label{general_SDE} \mathrm{d}Z_t=\beta(Z_t)\mathrm{d}t +\gamma(Z_t)\, \mathrm{d} B_t
 \ee
  has a unique weak solution with $Z_0=z_0$ for all
 $z_0 \in \R$. Let $F_0(x)$ be a probability distribution function on $\R$.

 Fix $k \in \N$.
 For each
 $n\in \N$  let $\left(\rd_n(T,i)\right)_{i \in [k], T \in \N}$ be a discrete time $\R^k$-valued stochastic process
  adapted to the filtration
 $\left(\cF_{n,T}\right)_{T \in \N}$. Let
\[d \rd_n(T,i):= \rd_n(T+1,i)-\rd_n(T,i).\]

   Suppose
 \begin{equation}\label{initial_dist_conv_okt23}
 \left( \rd_n(0,i) \right)_{i=1}^k \toind  \left(Z_{0,i}\right)_{i=1}^k \quad \text{ as } \quad n \to \infty
 \end{equation}
   where
 $\left(Z_{0,i}\right)_{i=1}^k$ are i.i.d. with distribution
 function $F_0$. Let $m: \N \to \N$ and
  suppose that for each $\mybart \in [0,+\infty)$ and each $1\leq i,j \leq
 k$ we have
 \begin{equation} \label{drift}
 \sup_{t \in \lbrack 0, \mybart \rbrack} \left|
 \sum_{T=0}^{\lfloor 2\cdot m(n) \cdot n\cdot t\rfloor} \condexpect{
 d \rd_n(T,i)}{\cF_{n,T}} - \frac{1}{2 m(n) \cdot n} \cdot  \sum_{T=0}^{\lfloor 2m(n) \cdot n \cdot t \rfloor}
 \beta( \rd_n(T,i) ) \right| \toinp 0
 \end{equation}
\begin{multline} \label{diffusion}
 \sup_{t \in \lbrack 0, \mybart \rbrack} \left|
 \sum_{T=0}^{\lfloor 2m(n)\cdot n \cdot t\rfloor} \condcov{
 d \rd_n(T,i) }{d \rd_n(T,j) }{\cF_{n,T}} - \right. \\
 \left.
 \frac{1}{2 m(n) \cdot n} \cdot
 \sum_{T=0}^{\lfloor 2m(n)\cdot n \cdot t \rfloor}
\ind [i=j]\cdot \gamma^2(\rd_n(t,i)) \right| \toinp 0
\end{multline}
 \begin{equation}
\label{tightness}
 \sum_{T=0}^{\lfloor 2m(n) \cdot n \cdot \mybart \rfloor}
\condexpect{( d \rd_n(T,i) )^2 \ind [
\abs{d \rd_n(T,i) }>\varepsilon ] }{\cF_{n,T}} \toinp 0 \text{
for all } \varepsilon>0
 \end{equation}
as $n \to \infty$.

Then the distributions of the $\R^k$-valued continuous-time stochastic processes
\[\left(  \rd_n \left( \lfloor 2m(n) \cdot n \cdot t \rfloor,i \right)   \right)_{i \in [k], t \geq 0}\]
 converge weakly to the
distribution of $\left( Z_{t,i} \right)_{i \in [k], t \geq 0}$ as $n \to
\infty$ in the Skorohod space $\D(\R^k)$,
 where $\left( Z_{t,i} \right)_{i \in [k], t \geq 0}$ are i.i.d. solutions of
\eqref{general_SDE},
 or briefly:
\begin{equation}\label{skorohod convergence}
\left(  \rd_n \left(\lfloor 2m(n) \cdot n \cdot t \rfloor,i \right)   \right)_{i \in [k], t \geq 0} \stackrel{ \mathcal{L}}{\longrightarrow} \left( Z_{t,i} \right)_{i \in [k], t \geq 0}
\end{equation}

\end{theorem}

$ $

\begin{proof}[Proof of Lemma \ref{lemma_degrees_converge_precise}]

$ $

We are going to use Theorem \ref{ispany_theorem} to prove that for all $k$ we have
\eqref{skorohod convergence}
 where $\left( Z_{t,i} \right)_{i \in [k], t \geq 0}$ are i.i.d. solutions of
\eqref{CIR_SDE_sketch} with initial distribution functions  $\prob{Z_{0,i} \leq x}=F_0(x)$, where $F_0(x)$ is defined as in 
Theorem \ref{theorem_fokszamok_fejlodese}. From this the claim of Lemma \ref{lemma_degrees_converge_precise} indeed follows,
 since by \eqref{edge_density_conv} we have $\lim_{n \to \infty} \frac{2m(n)}{n^2}=\rho(W)$, thus
 \[\left(  \rd_n \left(\lfloor 2m(n) \cdot n \cdot t \rfloor,i \right)   \right)_{i \in [k], t \geq 0}-
 \left(  \rd_n \left( \lfloor t \cdot  \rho(W) \cdot n^3
 \rfloor,i \right)   \right)_{i \in [k], t \geq 0} \stackrel{ \mathcal{L}}{\longrightarrow} \left( \, 0 \, \right)_{i \in [k], t \geq 0}, \]
  from which it follows that for each $t\geq 0$ the relation \eqref{degrees_convindist_iid_CIR} holds, where
 $\left( Z_{t,i} \right)_{i \in [k], t \geq 0}$ are i.i.d. solutions of
\eqref{CIR_SDE_sketch}, and using \eqref{atmenetsurusegfuggveny_CIR} we get that $\left( Z_{t,i}  \right)_{i \in [k]}$ are i.i.d. 
with distribution function $F_t(x)=\int_0^x f(t,y)\, \mathrm{d}y$ where
 $f(t,x)$ is defined by \eqref{CIR_densityfunction}.

\medskip

We need to check that \eqref{initial_dist_conv_okt23}, \eqref{drift}, \eqref{diffusion} and \eqref{tightness} holds
with \[\beta(z)=\kappa-\frac{\kappa}{\rho} z, \qquad \gamma(z)=\sqrt{2z}.\]

From the assumptions $\mathbf{X}_n(0) \toinp W$, \eqref{exponential_moment_initial_ketto} and Lemma
 \ref{lemma_uniform_integrabiliy_b}  it follows that
\[\left( \rd_n(0,i ) \right)_{i \in [k]} \toind
\left( D(W,U_i)\right)_{i \in [k]},\]
 thus by \eqref{degree_W_expect} and the definition of $F_0$ in Theorem \ref{theorem_fokszamok_fejlodese} we get that the 
probability distribution function of
 $D(W,U_i)$ is $F_0$ and 
 \eqref{initial_dist_conv_okt23} holds.

\medskip

Now we check that \eqref{drift} holds:
\begin{multline} \label{drift_int_calculation}
\sup_{t \in \lbrack 0, \mybart \rbrack} \left|
 \sum_{T=0}^{\lfloor 2\cdot m(n) \cdot n\cdot t\rfloor} \condexpect{
 d \rd_n(T,i)}{\cF_{n,T}} - \frac{1}{2 m(n) \cdot n} \cdot  \sum_{T=0}^{\lfloor 2m(n) \cdot n \cdot t \rfloor}
\left( \kappa-\frac{\kappa}{\rho} \rd_n(T,i) \right) \right|
\\
 \stackrel{\eqref{Vold}, \eqref{Vnew}, \eqref{degree_ind_evolution}}{\leq}
 \sum_{T=0}^{\lfloor 2m(n) \cdot n \cdot \mybart \rfloor}
 \left| \left(\frac{\rd_n(T,i)+\frac{\kappa}{n}}{2m(n)+n\kappa}
 -\frac{\rd_n(T,i)}{2m(n)} \right)-\frac{1}{2m(n)\cdot n} \left(\kappa
 -\frac{\kappa}{\rho} \rd_n(T,i) \right) \right|=\\
\sum_{T=0}^{\lfloor 2m(n) \cdot n \cdot \mybart \rfloor}
\frac{1}{2m(n)\cdot n}
\left( \left( \Ordo \left(\frac1n \right) +\left( \frac{\kappa}{\rho} -
\frac{\kappa}{ \frac{2 m(n)}{n^2}} \right) \right) \rd_n(T,i)   + \Ordo \left( \frac{n}{m(n)} \right) \right)
\end{multline}

By Lemma \ref{lemma_regularity} \eqref{lemma_i_degree_bound}  we have $\expect{\rd_n(T,i)}=\Ordo(1)$, thus
$\expect{\eqref{drift_int_calculation}} \to 0$ as $n \to \infty$
which implies \eqref{drift}.

\medskip

We prove \eqref{diffusion} by treating the cases $i=j$ and $i \neq j$ separately.

First we prove \eqref{diffusion} when $i=j$.  Using
 \eqref{Vold}, \eqref{Vnew}, \eqref{degree_ind_evolution}
 and the fact that $\cV_{new}(T)$
and $\cV_{old}(T)$ are conditionally independent given $\cF_{n,T}$ we get
 \begin{multline} \label{diffusion_int_calculation}
 \sup_{t \in \lbrack 0, \mybart \rbrack} \left|
 \sum_{T=0}^{\lfloor 2m(n)\cdot n \cdot t\rfloor} \condvar{
 d \rd_n(T,i) }{\cF_{n,T}} -
 \frac{1}{2 m(n) \cdot n} \cdot
 \sum_{T=0}^{\lfloor 2m(n)\cdot n \cdot t \rfloor}
 2 \rd_n(T,i) \right|
 \leq
 \\
\sum_{T=0}^{\lfloor 2m(n) \cdot n\mybart\rfloor} \left|
 \frac{1}{n^2}\left(
 \frac{n \rd_n(T,i)}{2m}\left(1-\frac{n \rd_n(T,i)}{2m}\right)+\right. \right. \\
\left. \left.
\frac{n \rd_n(T,i)+\kappa}{2m+n\kappa}\left(1-\frac{n \rd_n(T,i)+\kappa}{2m+n\kappa}\right)\right)
- \frac{1}{2m(n)\cdot n} 2\rd_n(T,i)  \right|=\\
\sum_{T=0}^{\lfloor 2m(n) \cdot n\mybart\rfloor}
\frac{1}{2 m(n) \cdot n} \left(
\Ordo \left( \frac{n}{m(n)} \right) \rd_n(T,i)^2 +\Ordo \left( \frac{1}{n} \right) +
\Ordo \left( \frac{n}{m(n)} \right) \rd_n(T,i)
\right)
 \end{multline}
 By Lemma \ref{lemma_regularity} \eqref{lemma_i_degree_bound} we have 
 $\expect{\rd_n(T,i)^2}=\Ordo(1)$, thus
$\expect{\eqref{diffusion_int_calculation}} \to 0$ as $n \to \infty$ which implies
\eqref{diffusion} for $i=j$.

Now we prove \eqref{diffusion} when $i \neq j$:
\begin{multline}\label{cov_int_calculation}
 \sup_{t \in \lbrack 0, \mybart
\rbrack} \left|
 \sum_{T=0}^{\lfloor 2m(n)\cdot  n \cdot t\rfloor} \condcov{
 d \rd_n(T,i) }{ d \rd_n(T,j) }{\cF_{n,T}} \right|
 \leq\\
 \sum_{T=0}^{\lfloor 2m(n)\cdot n \cdot \mybart\rfloor}
\left| - \left(\frac{\rd_n(T,i)}{2m(n)} \cdot \frac{\rd_n(T,j)+\frac{\kappa}{n}}{2m(n)+n\kappa}
+ \frac{\rd_n(T,j)}{2m(n)} \cdot \frac{\rd_n(T,i)+\frac{\kappa}{n}}{2m(n)+n\kappa}\right)
 -\right. \\
\left.
 \left(\frac{\rd_n(T,i)+\frac{\kappa}{n}}{2m(n)+n\kappa}
 -\frac{\rd_n(T,i)}{2m(n)} \right) \cdot \left(\frac{\rd_n(T,j)+\frac{\kappa}{n}}{2m(n)+n\kappa}
 -\frac{\rd_n(T,j)}{2m(n)} \right)
   \right|=\\
\sum_{T=0}^{\lfloor 2m(n)\cdot n \cdot \mybart\rfloor}
\frac{1}{2 m(n) \cdot n} \left( \Ordo \left( \frac{n}{m(n)} \right) \left( \rd_n(T,i)+\rd_n(T,j) \right)^2 \right)
\end{multline}
By  Lemma \ref{lemma_regularity} \eqref{lemma_i_degree_bound} we have $\expect{\rd_n(T,i)^2}=\Ordo(1)$ and
 $\expect{\rd_n(T,j)^2}=\Ordo(1)$, which implies
$\expect{\eqref{cov_int_calculation}} \to 0$ as $n \to \infty$ which in turn implies
\eqref{diffusion} for $i\neq j$.
\eqref{tightness} is trivial since
$\prob{\abs{d \rd_n(T,i) }\leq \frac{1}{n}}=1$.

Having checked that \eqref{initial_dist_conv_okt23}, \eqref{drift}, \eqref{diffusion} and \eqref{tightness} holds,
  we can use Theorem \ref{ispany_theorem} to prove that we have
\eqref{skorohod convergence}
 where $\left( Z_{t,i} \right)_{i \in [k], t \geq 0}$ are i.i.d. solutions of
\eqref{CIR_SDE_sketch} with initial distribution functions  $F_0(x)$, which finishes the proof of Lemma
\ref{lemma_degrees_converge_precise}, as described in the beginning of the proof.
\end{proof}

\subsection{Asymptotic edge-stationarity}
\label{subsection_elatkotos_proof_of_edgestationarity}

Similarly to Section \ref{section_proof_of_thm_elek} we assume that the distribution of $\mathbf{X}_n(T)$ is vertex exchangeable.
We are going to prove \eqref{thm_fokok_statement_conv_graphlim} using Lemma \ref{lemma_homkonv_konv_in_dist}: 
we only need to show that for all $k \in \N$ and
 $t > 0$ we have
 \begin{equation}\label{thm_degrees_formula_convindist}
  \mathbf{X}_n^{[k]} \left( \lfloor t \cdot  \rho(W) \cdot n^3
 \rfloor \right)  \toind \mathbf{X}_{\hat{W}_t}^{[k]}.
 \end{equation}

\begin{proof}[Proof of Theorem \ref{theorem_fokszamok_fejlodese}]
 Let $\left( Z_{t,i}  \right)_{i \in [k]}$ denote i.i.d. random variables with 
distribution function $F_t(x)=\int_0^x f(t,y)\, \mathrm{d}y$ where
 $f(t,x)$ is defined by \eqref{CIR_densityfunction}. Recall the notion of $\mypoi(k, \lambda)$ from \eqref{def_mypoi}.
 Define the function $\mypoi(A, \left( z_{i}\right)_{i=1}^k)$ for $A \in \A_k$ and $z_i \in [0,+\infty)$, $i \in [k]$ by
\begin{equation}\label{poi_graphlim_W_formula_okt25}
\mypoi(A,\left(z_i\right)_{i=1}^k):=
\prod_{i=1}^k \prod_{j=i}^k
\mypoi \left( A^*(i,j), \frac{ z_{i} \cdot z_{j} }{ \rho \cdot (1+\ind[i=j]) } \right).
\end{equation}
By \eqref{X_W_indep_prod_formula} and \eqref{fokszamok_fejlodese_grafon}, in order to prove 
\eqref{thm_degrees_formula_convindist} we only need to check that
for all $A \in \A_k$ 
\begin{equation*}
\lim_{n \to \infty} \prob{ \mathbf{X}_n^{[k]} \left( \lfloor t \cdot  \rho(W) \cdot n^3
 \rfloor \right)= A}= \expect{ \mypoi(A, \left( Z_{t,i}\right)_{i=1}^k)}.
\end{equation*}

We (somewhat arbitrarily) fix  $2< \nu <\frac{5}{2}$. Let 
\[T_0^n:=\lfloor t \cdot  \rho(W) \cdot n^3 \rfloor - \lfloor n^{\nu}
 \rfloor.\]
It easily follows from $\nu <\frac{5}{2}< 3$,  Lemma
\ref{lemma_regularity} \eqref{lemma_iii_degree_change} and Lemma \ref{lemma_degrees_converge_precise}  that
\begin{equation}\label{degrees_convindist_iid_CIR_korabb}
 \left( \rd_n( T_0^n,i ) \right)_{i\in[k]} \toind \left( Z_{t,i}  \right)_{i \in [k]}  \quad \text{ as } \quad n \to \infty.
 \end{equation}

Now we couple
$\mathbf{X}_n^{[k]}(T_0^n+T)$
to $\mathbf{Y}_n^{[k]}\left(\frac{T}{m(n)}\right)$ in a similar fashion as in Section \ref{section_proof_of_thm_elek}:
\begin{itemize}
 \item The initial state of $\mathbf{Y}_n^{[k]}$ is $\forall \, i,j \in [k]:\; Y_n(0,i,j)=X_n(T_0^n,i,j)$.
\item Given $\mathbf{X}_n(T_0^n)$, 
 the entries $\left(Y_n(t,i,j)\right)_{i\leq j \leq k}$ evolve independently and
 $Y_n(t,i,j)\equiv Y_n(t,j,i)$.
\item Given $\mathbf{X}_n(T_0^n)$, the evolution
 of $Y^*_n(t,i,j)$ is is an M/M/$\infty$-queue
  with service rate $1$ and arrival rate
\begin{equation}\label{MMinfty_rate_T_null}
 \mu=\mu_{i,j}:=\frac{ d(\mathbf{X}_n(T_0^n),i)d(\mathbf{X}_n(T_0^n),j)}{2m(n)\cdot (1+\ind[i=j]) }=
 \frac{ \rd_n(T_0^n,i) \rd_n(T_0^n,j)}{ \frac{2m(n)}{n^2} \cdot (1+\ind[i=j]) }.
\end{equation}
\end{itemize}

Now we show that
 \begin{equation}\label{Y_tablazat_konv_eo_stac}
  \mathbf{Y}_n^{[k]} \left( \frac{ \lfloor n^{\nu} \rfloor}{m(n)} \right)  
\toind \mathbf{X}_{\hat{W}_t}^{[k]} \quad \text{ as } \quad n \to \infty.
 \end{equation}
First note that
\begin{multline}\label{duplalim_formula}
\lim_{n \to \infty} \lim_{s \to \infty} \prob{ \mathbf{Y}_n^{[k]}( s )=A}\stackrel{\eqref{queue_t_0_infty}}{=}
\lim_{n \to \infty} \expect{ \prod_{i=1}^k \prod_{j=i}^k
\mypoi \left( A^*(i,j), \frac{ \rd_n(T_0^n,i) \rd_n(T_0^n,j)}{ \frac{2m(n)}{n^2} \cdot (1+\ind[i=j]) } \right)}
\stackrel{\eqref{degrees_convindist_iid_CIR_korabb} }{=}\\
\expect{ \prod_{i=1}^k \prod_{j=i}^k
\mypoi \left( A^*(i,j), \frac{ Z_{t,i} \cdot Z_{t,j} }{ \rho \cdot (1+\ind[i=j]) } \right)}
\stackrel{\eqref{poi_graphlim_W_formula_okt25}}{=}
\expect{ \mypoi(A, \left( Z_{t,i}\right)_{i=1}^k)}
\end{multline}
Let $t_n:=\frac{ \lfloor n^{\nu} \rfloor}{m(n)}$. $\lim_{n \to \infty} t_n=+\infty$ follows from  $2<\nu$. We have
\begin{multline}\label{annealed_conv_to_stac}
\abs{ \prob{ \mathbf{Y}_n^{[k]}( t_n )=A} -
\lim_{s \to \infty} \prob{ \mathbf{Y}_n^{[k]}( s )=A}} \stackrel{\eqref{lemma_kicsi_eq}}{\leq} \\
\exp(-t_n) \cdot
\sum_{i=1}^k \sum_{j=i}^k
\left(
\expect{ X_n(T_0^n,i,j)}+
\expect{\mu_{i,j} }
\right)
\stackrel{\eqref{MMinfty_rate_T_null}, \eqref{exp_lecseng_a_foxam}, \eqref{moment_p_bound}}{=}
\exp(-t_n)\Ordo(1).
\end{multline}
Thus \eqref{Y_tablazat_konv_eo_stac} follows from \eqref{duplalim_formula} and \eqref{annealed_conv_to_stac}.

Using the proof of \eqref{szoros_a_csat} we can  construct a coupling such that we have
\begin{equation*}
\lim_{n \to \infty}
\prob{ \forall \, 0 \leq T \leq n^{\nu}: \quad  \mathbf{X}_n^{[k]}(T_0^n+T) =
 \mathbf{Y}_n^{[k]} \left( \frac{T}{m(n)} \right) } =1.
\end{equation*}
Now \eqref{thm_degrees_formula_convindist} follows from this, $T_0^n+ \lfloor n^{\nu} \rfloor=
\lfloor t \cdot  \rho(W) \cdot n^3 \rfloor$ and \eqref{Y_tablazat_konv_eo_stac}.
\end{proof}

\end{document}